\documentclass[a4paper,reqno,11pt]{amsart}

\usepackage{syntonly}

\usepackage[american]{babel}

\usepackage[utf8]{inputenc}
\usepackage[T1]{fontenc}
\usepackage{microtype}
\usepackage[a4paper, scale={0.72,0.74}, marginratio={1:1}, footskip=7mm, headsep=10mm]{geometry}

\usepackage{indentfirst}

\usepackage{emptypage}  

\usepackage{perpage}

\usepackage{titlesec}
\titleformat*{\section}{\scshape\Large\bfseries}
\titleformat*{\subsection}{\scshape\large\bfseries}

\usepackage{enumitem}

\usepackage{fixltx2e}

\usepackage{color}

\newcommand\red{\textcolor{red}}

\usepackage{multicol}

\usepackage{hyperref}

\usepackage{booktabs}
\usepackage{graphicx}
\usepackage{tabularx}
\usepackage{caption,subcaption}
\usepackage{longtable}
\usepackage[vcentermath]{youngtab}
\usepackage{tikz,tikz-cd}
\usetikzlibrary{matrix, calc, arrows, decorations.markings, positioning}

\usepackage{amsmath,amssymb,amsthm}
\usepackage{stmaryrd,mathtools,mathdots}
\usepackage{eufrak}
\usepackage{wasysym}

\usepackage{accents}

\usepackage{bm,bbm}

\usepackage{braket}  

\numberwithin{equation}{section}

\usepackage{subdepth}

\makeatletter
\renewcommand{\@secnumfont}{\bfseries}
 
\renewcommand\section{\@startsection{section}{1}%
\z@{.7\linespacing\@plus\linespacing}{.5\linespacing}%
{\large\scshape\bfseries\centering}}

\renewcommand\subsection{\@startsection{subsection}{2}%
  \z@{.5\linespacing\@plus.7\linespacing}{-.5em}%
  {\bfseries\scshape}}
\makeatother

\newenvironment{myitemize}{%
\begin{list}{$\bullet$}%
 	{%
	\setlength{\itemsep}{0.4em}%
	\setlength{\topsep}{0.5em}%
	\setlength\leftmargin{2.45em}%
	\setlength\labelwidth{2.05em}%
	\setlength{\labelsep}{0.4em}%
	}%
	}%
{\end{list}}

{\end{myitemize}}

\MakePerPage{footnote}

\makeatletter
\newcommand*{\myfnsymbolsingle}[1]{%
  \ensuremath{%
    \ifcase#1% 0
    \or % 1
      \dagger
    \else % >= 2
      \@ctrerr  
    \fi
  }%   
}   
\makeatother

\usepackage{alphalph}
\newalphalph{\myfnsymbolmult}[mult]{\myfnsymbolsingle}{}

\theoremstyle{plain}
\newtheorem{theorem}{Theorem}[section]

\newtheorem{proposition}[theorem]{Proposition}
\newtheorem{corollary}[theorem]{Corollary}

\theoremstyle{definition}

\DeclarePairedDelimiter\abs{\lvert}{\rvert} % absolute value
\makeatletter
\let\oldabs\abs
\def\abs{\@ifstar{\oldabs}{\oldabs*}}

\DeclarePairedDelimiterX{\norm}[1]{\lVert}{\rVert}{#1} % norm
\let\oldnorm\norm
\def\norm{\@ifstar{\oldnorm}{\oldnorm*}}

\DeclarePairedDelimiterX{\ceil}[1]{\lceil}{\rceil}{#1} % ceiling
\let\oldceil\ceil
\def\ceil{\@ifstar{\oldceil}{\oldceil*}}

\DeclarePairedDelimiterX{\floor}[1]{\lfloor}{\rfloor}{#1} % floor
\let\oldfloor\floor
\def\floor{\@ifstar{\oldfloor}{\oldfloor*}}
\makeatother

\newcommand{\ga}{\alpha}
\newcommand{\gb}{\beta}
\newcommand{\diff}{\mathop{}\!\mathrm{d}}

\renewcommand{\P}{\mathbb{P}} % probability
\newcommand{\1}{\mathbbm{1}} % indicator function
\renewcommand{\i}{\mathrm{i}} % imaginary unit
\DeclareMathOperator{\e}{\mathrm{e}} % Euler number
\DeclareMathOperator{\GL}{GL} % group GL
\DeclareMathOperator{\SO}{SO} % group SO
\DeclareMathOperator{\schur}{s} % Schur fn
\renewcommand{\sp}{\operatorname{sp}} % sp Schur fn
\DeclareMathOperator{\Ai}{Ai} % Airy function

\newcommand{\N}{\mathbb{N}} % natural numbers
\newcommand{\R}{\mathbb{R}} % real numbers
\newcommand{\C}{\mathbb{C}} % complex numbers
\renewcommand{\epsilon}{\varepsilon}
\renewcommand{\rho}{\varrho}
\renewcommand{\phi}{\varphi}

\DeclareMathSymbol{\widehatsym}{\mathord}{largesymbols}{"62}

\renewcommand{\tilde}{\widetilde} % wider tilde

\makeatletter
\newcommand{\doubletilde}[1]{{%
  \mathpalette\double@tilde{#1}%
}}
\newcommand{\double@tilde}[2]{%
  \sbox\z@{$\m@th#1\tilde{#2}$}%
  \ht\z@=.9\ht\z@
  \tilde{\box\z@}%
}
\makeatother

\newenvironment{myenumerate}{%
\renewcommand{\theenumi}{(\roman{enumi})}%
\renewcommand{\labelenumi}{\theenumi}%
\begin{list}{\labelenumi}
	{%
	\setlength{\itemsep}{0.4em}%
	\setlength{\topsep}{0.5em}%
	\setlength\leftmargin{2.45em}%
	\setlength\labelwidth{2.05em}%
	\setlength{\labelsep}{0.4em}%
	\usecounter{enumi}%
	}%
	}%
{\end{list}
}
\renewenvironment{enumerate}{
\begin{myenumerate}}%
{\end{myenumerate}}

  {\left\lbrace\begin{array}{@{}l@{}}}%
  {\end{array}\right.}

\makeatletter
\newsavebox{\mybox}\newsavebox{\mysim}
\newcommand{\distras}[1]{%
  \savebox{\mybox}{\hbox{\kern3pt$\scriptstyle#1$\kern3pt}}%
  \savebox{\mysim}{\hbox{$\sim$}}%
  \mathbin{\overset{#1}{\kern\z@\resizebox{\wd\mybox}{\ht\mysim}{$\sim$}}}%
}
\makeatother

\DeclareMathOperator{\Flat}{flat}
\DeclareMathOperator{\hFlat}{h-flat}

\newcommand{\fPi}{\Pi^{\Flat}} %flat path
\newcommand{\hPi}{\Pi^{\hFlat}} %half-flat path
\newcommand{\fTau}{\tau^{\Flat}} %flat LPP
\newcommand{\hTau}{\tau^{\hFlat}} %half-flat LPP
\newcommand{\fK}{K^{\Flat}} % flat K
\newcommand{\hK}{K^{\hFlat}} %half- flat K
\newcommand{\fKone}{K^{\Flat ,1}}
\newcommand{\fKtwo}{K^{\Flat ,2}}
\newcommand{\fKthree}{K^{\Flat ,3}}
\newcommand{\fKfour}{K^{\Flat ,4}}
\newcommand{\fKonetilde}{\tilde{K}^{\Flat ,1}}
\newcommand{\fKtwotilde}{\tilde{K}^{\Flat ,2}}
\newcommand{\fKthreetilde}{\tilde{K}^{\Flat ,3}}
\newcommand{\fKfourtilde}{\tilde{K}^{\Flat ,4}}

\newcommand{\hKone}{K^{\hFlat ,1}}
\newcommand{\hKtwo}{K^{\hFlat ,2}}
\newcommand{\hKthree}{K^{\hFlat ,3}}

\newcommand{\hKonetilde}{\tilde{K}^{\hFlat ,1}}
\newcommand{\hKtwotilde}{\tilde{K}^{\hFlat ,2}}
\newcommand{\hKthreetilde}{\tilde{K}^{\hFlat ,3}}

\newcommand{\fKtilde}{\tilde{K}^{\Flat}} % flat K
\newcommand{\hKtilde}{\tilde{K}^{\hFlat}} %half- flat K
\newcommand{\fHbar}{\widehat{H}^{\Flat}} % flat K
\newcommand{\hHbar}{\widehat{H}^{\hFlat}} %half- flat K
\newcommand{\fH}{H^{\Flat}} % flat H
\newcommand{\hH}{H^{\hFlat}} %half- flat H
\begin{document}

\title[GOE, ${\rm Airy}_{ 2\to 1}$ via symplectic Schur functions]{GOE and ${\rm Airy}_{2\to 1}$ marginal distribution \\ via symplectic Schur functions}

\author[E.~Bisi]{Elia Bisi}
\address{Department of Statistics\\
University of Warwick\\
Coventry CV4 7AL, UK}
\email{E.Bisi@warwick.ac.uk}

\author[N.~Zygouras]{Nikos Zygouras}
\address{Department of Statistics\\
University of Warwick\\
Coventry CV4 7AL, UK}
\email{N.Zygouras@warwick.ac.uk}


\maketitle

\begin{center}
{\it To Raghu Varadhan with great respect on the occasion of his $75th$ birthday}
\end{center}
\vskip 4mm

\begin{abstract}
We derive Sasamoto's Fredholm determinant formula for the Tracy-Widom GOE distribution, as well as the one-point marginal distribution of the ${\rm Airy}_{2\to1}$ process, originally derived by Borodin-Ferrari-Sasamoto, as scaling limits of point-to-line and point-to-half-line directed last passage percolation with exponentially distributed waiting times.
The asymptotic analysis goes through new expressions for the last passage times in terms of integrals of (the continuous analog of) symplectic and classical Schur functions, obtained recently in~\cite{BZ19a}.
\end{abstract}

\tableofcontents

\addtocounter{section}{0}

\section{Introduction}
\label{sec:model}

The goal of this contribution is to provide a new route to the Tracy-Widom GOE distribution and to the one-point marginal of the ${\rm Airy}_{2\to 1}$ process through asymptotics of directed last passage percolation (dLPP), starting from new exact formulae involving not only the usual Schur functions (that have been appearing so far in treatments of dLPP) but also symplectic Schur functions.
The first systematic study on the dLPP model goes back to Johansson~\cite{J00}, who derived the Tracy-Widom GUE limiting distribution for the model with point-to-point path geometry and geometrically distributed waiting times. For an overview on dLPP and other similar integrable probabilistic models related to random matrix theory and determinantal structures, the reader is referred to~\cite[ch.~10]{F10} and~\cite{K05}.
\vskip 1mm
Let us start by introducing the details of the directed last passage percolation model that we will work with.
 On the $(1+1)$-dimensional lattice we consider paths with fixed starting point and with ending point  lying free on a line or half-line. 
 We also consider a random field on the lattice, which is a collection $\{W_{i,j}\colon (i,j)\in \N^2\}$ of independent random variables, usually called \emph{weights} or \emph{waiting times} and which we will assume to be exponentially distributed. We will consider two different path geometries (see Fig.~\ref{fig:directedPath} for a graphical representation):
\begin{enumerate}
\item
The \emph{point-to-line} directed last passage percolation, defined by
\begin{equation}
\label{eq:fLPP}
\fTau_N := \max_{\pi\in\fPi_N} \sum_{(i,j)\in \pi} W_{i,j} \, ,
\end{equation}
where $\fPi_N$ is the set of directed (down-right) lattice paths of length $N$ (namely, made up of $N$ vertices) starting from $(1,1)$ and ending at any point of the line $\{(i,j)\in\N^2 \colon i+j = N+1\}$.
\item
The \emph{point-to-half-line} directed last passage percolation, defined by
\begin{equation}
\label{eq:hLPP}
\hTau_N := \max_{\pi\in\hPi_N} \sum_{(i,j)\in \pi} W_{i,j} \, ,
\end{equation}
where $\hPi_N$ is the set of directed lattice paths of length $N$ starting from $(1,1)$ and ending on the half-line $\{(i,j)\in\N^2 \colon i+j=N+1\, , \,\, i\leq j\}$.
\end{enumerate}

\begin{figure}
\centering
\begin{minipage}[b]{.5\linewidth}
\centering
\begin{tikzpicture}[scale=0.5]
\draw[thick,dashed] (10.5,-.5) -- (.5,-10.5);
\foreach \i in {1,...,10}{
	\draw[thick] (1,-\i) grid (11-\i,-\i);
	\draw[thick] (\i,-1) grid (\i,-11+\i);
	
		\foreach \j in {\i,...,10}{
		\node[draw,circle,inner sep=1pt,fill] at (11-\j,-\i) {};
	}
}

\draw[thick,color=red,-] (1,-1) -- (2,-1) -- (2,-2) -- (3,-2) -- (3,-3) -- (3,-4) -- (3,-5) -- (4,-5) -- (5,-5) -- (5,-6);
\foreach \x in {(1,-1),(2,-1),(2,-2),(3,-2),(3,-3),(3,-4),(3,-5),(4,-5),(5,-5),(5,-6)}{
	\node[draw,circle,inner sep=1pt,fill,red] at \x {};
	}
\end{tikzpicture}
\subcaption{Point-to-line path}
\label{subfig:P2Lpath}
\end{minipage}%
\begin{minipage}[b]{.5\linewidth}
\centering
\begin{tikzpicture}[scale=0.5]
\draw[thick,dashed] (10.5,-.5) -- (5.5,-5.5); \draw[thick,dashed] (1,-1) -- (6,-6);
\foreach \i in {1,...,5}{
	\draw[thick] (1,-\i) grid (11-\i,-\i);
	
		\foreach \j in {\i,...,10}{
		\node[draw,circle,inner sep=1pt,fill] at (11-\j,-\i) {};
	}
}

\foreach \j in {1,...,5}{
	\draw[thick] (\j,-1) grid (\j,-5);
	\draw[thick] (5+\j,-1) grid (5+\j,-6+\j);
}

\draw[thick,color=red,-] (1,-1) -- (1,-2) -- (1,-3) -- (2,-3) -- (3,-3) -- (4,-3) -- (5,-3) -- (6,-3) -- (6,-4) -- (7,-4);
\foreach \x in {(1,-1),(1,-2),(1,-3),(2,-3),(3,-3),(4,-3),(5,-3),(6,-3),(6,-4),(7,-4)}{
	\node[draw,circle,inner sep=1pt,fill,red] at \x {};
	}
\end{tikzpicture}
\subcaption{Point-to-half-line path}
\label{subfig:P2half-Lpath}
\end{minipage}%
\caption{Directed paths in $\N^2$ of length $10$ from the point $(1,1)$ to the line $i+j-1=10$. The paths, highlighted in red, correspond to the two geometries specified. The picture is rotated by $90^{\circ}$ clockwise w.r.t.\ the Cartesian coordinate system, to adapt it to the usual matrix/array indexing.}
\label{fig:directedPath}
\end{figure}
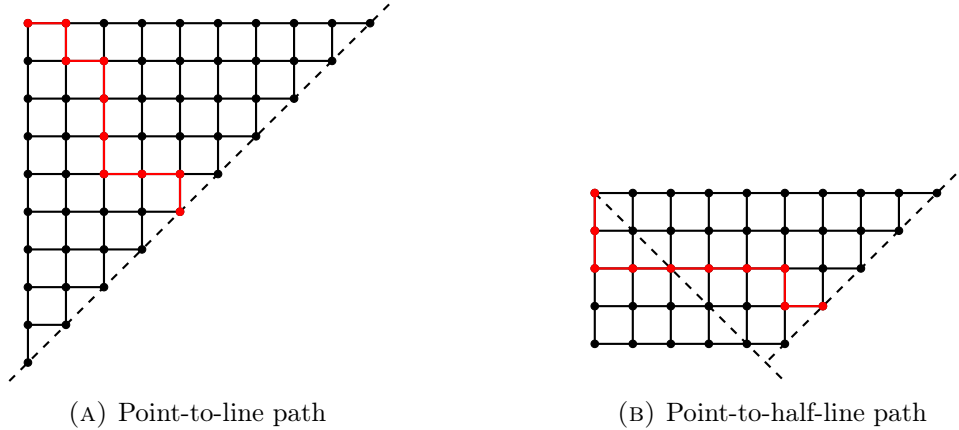

As already mentioned, in our treatment we will assume the weights $W_{i,j}$'s to be exponentially distributed. More specifically, for the point-to-line dLPP $\fTau_{2N}$, we consider a triangular array of independent weights $\bm{W}=\{W_{i,j}\colon (i,j)\in\N^2 \, , \,\, i+j\leq 2N+1 \}$ distributed as
\begin{equation}
\label{eq:expDistribution}
W_{i,j} \sim
\begin{cases}
{\rm Exp}(\alpha_i + \beta_j) &1\leq i,j\leq N \, , \\
{\rm Exp}(\alpha_i + \alpha_{2N-j+1}) &1\leq i\leq N \, , \,\, N < j\leq 2N-i+1 \, , \\
{\rm Exp}(\beta_{2N-i+1} + \beta_j) &1\leq j\leq N \, , \,\, N < i \leq 2N-j+1 \, ,
\end{cases}
\end{equation}
for fixed parameters $\alpha_1,\dots,\alpha_{N},\beta_1,\dots,\beta_{N}\in\R_+$.
Similarly, for the point-to-half-line dLPP $\hTau_{2N}$, we consider a trapezoidal array of independent weights $\bm{W}=\{W_{i,j}\colon (i,j)\in\N^2 \, , \,\, i+j\leq 2N+1 \, , \,\, i\leq N \}$ distributed according to~\eqref{eq:expDistribution} for $i\leq N$.
\vskip 1mm
Let us now present our results starting with the flat case. We will prove that $\fTau_{2N}$ has fluctuations of order $N^{1/3}$, and its scaling limit is given by the Tracy-Widom GOE distribution, whose cumulative distribution function $F_1$ can be expressed (among other expressions) via a Fredholm determinant \cite{Sa05} as 
\begin{equation}
\label{eq:GOE}
F_1(s)
= \det(I - \mathcal{K}_1)_{L^2([s,\infty))}
\end{equation}
for $s\in\R$, where $\mathcal{K}_1$ is the operator on $L^2([s,\infty))$ defined through the kernel
\begin{align}
\label{GOEkernel}
\mathcal{K}_1(\lambda,\xi) := \frac{1}{2}\Ai\bigg(\frac{\lambda + \xi}{2}\bigg) \, .
\end{align}
\begin{theorem}
\label{thm:fAsympt}
If the waiting times are independent and exponentially distributed with rate $2\gamma$, the limiting distribution of the point-to-line directed last passage percolation $\fTau_{2N}$ is given, for $r\in\R$, by
\begin{equation}
\label{eq:fAsympt}
\lim_{N\to\infty}\P\bigg(\fTau_{2N} \leq \frac{2N}{\gamma} + rN^{1/3} \bigg) 
= F_1\big(2^{1/3}\gamma r\big) \, .
\end{equation}
\end{theorem}

Regarding the point-to-half-line dLPP, we will prove that, for exponentially distributed i.i.d.\  environment, $\hTau_{2N}$ has fluctuations of order $N^{1/3}$, and its scaling limit is given by the one-point distribution $F_{2\to 1}$ of the $\rm{Airy}_{2\to 1}$ process. The expression we will arrive at is the following~\cite{BFS08}:
\begin{equation}
\label{eq:Airy21}
F_{2 \to 1}(s)
= \det(I - \mathcal{K}_{2\to 1})_{L^2([s,\infty))} \, ,
\end{equation}
where $\mathcal{K}_{2\to 1}$ is the operator on $L^2([s,\infty))$ defined through the kernel
\begin{equation}
\label{eq:Airy21kernel}
\mathcal{K}_{2\to 1}(\lambda,\xi)
:= \int_0^{\infty} \Ai(\lambda+x) \Ai(\xi+x) \diff x
+ \int_0^{\infty} \Ai(\lambda+x) \Ai(\xi-x) \diff x
\, .
\end{equation}
\begin{theorem}
\label{thm:hAsympt}
If the waiting times are independent and exponentially distributed with rate $2\gamma$, the limiting distribution of the point-to-half-line directed last passage percolation $\hTau_{2N}$ is given, for $r\in\R$, by
\[
\lim_{N\to\infty}\P\bigg(\hTau_{2N} \leq \frac{2N}{\gamma} + rN^{1/3} \bigg) 
= F_{2\to 1}\big(2^{-1/3} \gamma r \big) \, .
\]
\end{theorem}
\vskip 4mm
Expression~\eqref{eq:GOE}-\eqref{GOEkernel} for the GOE distribution is different from the one originally derived by Tracy and Widom, first expressed in terms of Painlev\'e functions~\cite{TW96} and then in terms of a block-Fredholm Pfaffian~\cite{TW98, TW05}. 
Sasamoto's original derivation of~\eqref{GOEkernel} came through the analysis of the Totally Asymmetric Simple Exclusion Process (TASEP), with an initial condition of the form $\cdots0101010000\cdots$, where $1$ denotes a particle and $0$ a hole. The presence of the semi infinite sequence of holes is technical and the actual focus of the asymptotic analysis in~\cite{Sa05} is on the alternating particle-hole regime, which simulates \emph{flat} initial conditions.
  The starting point for this derivation in \cite{Sa05} was Sch\"utz's determinantal formula~\cite{S97} for the occupation probabilites in TASEP, obtained via Bethe ansatz methods. A proof that Sasamoto's formula actually provides a different expression for the Tracy-Widom GOE distribution was provided in~\cite{FS05}. Subsequently to~\cite{Sa05}, the $F_{2\to 1}$ distribution was derived in~\cite{BFS08} by studying, again via Sch\"utz's and Sasamoto's formulae, the asymptotic distribution of TASEP particles with initial configuration $\cdots0101010000\cdots$, but now at the interface between the right half end $000\cdots$ and the left half alternating (flat) configuration $\cdots010101$. 
  \vskip 2mm
  Asymptotics recovering the Tracy-Widom GOE distribution as a limiting law have been performed in~\cite{BR01, F04} for directed last passage percolation and polynuclear growth models, and more recently in~\cite{LeDC12}, at a nonrigorous level,  for the KPZ equation with flat initial data. All these works derive Painlev\'e expressions or various forms of block-Fredholm Pfaffian formulae for the Tracy-Widom GOE. In contrast, our approach leads directly to Sasamoto's Fredholm determinant formula~\eqref{eq:GOE}-\eqref{GOEkernel}, as well as to the ${\rm Airy_{2\to1 }}$ formula~\eqref{eq:Airy21}-\eqref{eq:Airy21kernel}.
\vskip 2mm
Methodologically, both \cite{BR01, F04} use a symmetrization argument that amounts to considering point-to-point dLPP on a square array with waiting times symmetric along the antidiagonal. We do not use such an
argument but we rather start from the exact integral formulae obtained in~\cite{BZ19a} for the cumulative distribution function of~\eqref{eq:fLPP} and~\eqref{eq:hLPP} with the exponential distribution specified above, in terms of integrals of (the continuous analog of) both standard and symplectic Schur functions, cf.~\eqref{eq:fLPPint}, \eqref{eq:hLPPint}.
In representation theory, Schur functions appear as characters of irreducible representations of classical groups~\cite{FH91} (in our specific case, general linear groups and symplectic groups), and are defined as certain sums on Gelfand-Tsetlin patterns. Continuous Schur functions, which are of our interest here, are continuous limits of rescaled Schur functions, and by Riemann sum approximation turn out to be integrals on continuous Gelfand-Tsetlin patterns (see for example~\cite[Prop.~4.2]{BZ19a} for the symplectic case).
Due to the determinantal structure of Schur functions, which in turn arises from the Weyl character formula~\cite{FH91}, continuous Schur functions can also be shown to have a determinantal form. The continuous analogs of standard and symplectic Schur functions, denoted by $\schur^{\rm cont}_{\bm{\alpha}}(\bm{x})$ and $\sp^{\rm cont}_{\bm{\alpha}}(\bm{x})$ respectively for $\bm{\alpha},\bm{x}\in\R^N$, thus have the following representation:
\begin{align}
\label{eq:contS}
\schur^{\rm cont}_{\bm{\alpha}}(\bm{x})
&= \frac{\det(\e^{\alpha_j x_i})_{1\leq i,j\leq N}}
{\prod_{1\leq i<j\leq N} (\alpha_i - \alpha_j)} \, , \\
\label{eq:contSp}
\sp^{\rm cont}_{\bm{\ga}}(\bm{x})
&= \frac{\det\big(\e^{\alpha_j x_i} - \e^{-\alpha_j x_i}\big)_{1\leq i,j\leq N}}
{\prod_{1\leq i<j\leq N}(\alpha_i-\alpha_j) \prod_{1\leq i\leq j\leq N}(\alpha_i + \alpha_j)} \, .
\end{align}
Our formulae for the two dLPP models then read as follows.
\begin{proposition}[\cite{BZ19a}]
Ignoring the normalization constants that depend on the parameters $\alpha_i$'s and $\beta_j$'s only, the distribution functions of $\fTau_{2N}$ and $\hTau_{2N}$ are given by
\begin{align}
\label{eq:fLPPint}
\P(\fTau_{2N}\leq u)
&\propto \e^{-u \sum_{k=1}^N( \alpha_k+\beta_k)}
\int_{\{0\leq x_N\leq \cdots \leq x_1\leq u\}}  
\sp^{\rm cont}_{\bm{\ga}}(\bm{x})  \sp^{\rm cont}_{\bm{\gb}}(\bm{x}) \prod_{i=1}^N \diff x_i \, , \\
\label{eq:hLPPint}
\P(\hTau_{2N} \leq u)
&\propto \e^{-u \sum_{k=1}^N( \alpha_k+\beta_k)}
\int_{\{0\leq x_N\leq \cdots \leq x_1\leq u\}}  
\sp^{\rm cont}_{\bm{\alpha}}(\bm{x})  \schur^{\rm cont}_{\bm{\beta}}(\bm{x}) \prod_{i=1}^N \diff x_i \, .
\end{align}
\end{proposition}
Using the determinantal form of Schur functions~\eqref{eq:contS} and~\eqref{eq:contSp} and the well-known Cauchy-Binet identity, which expresses the multiple integral of two determinantal functions as the determinant of an integral, formulae~\eqref{eq:fLPPint} and~\eqref{eq:hLPPint} can be turned to ratios of determinants and then to Fredholm determinants,  amenable to asymptotic analysis via steepest descent.
\vskip 1mm
It is worth noting that the above formulae have been derived in~\cite{BZ19a} as the zero temperature limit of certain integrals of Whittaker functions that represent the Laplace transform of the log-gamma polymer partition function~\cite{Sep12, COSZ14, OSZ14} 
in the same path geometries.
The derivation of the Laplace transforms for the log-gamma polymer in~\cite{BZ19a} used combinatorial arguments through the geometric Robinson-Schensted-Knuth correspondence~\cite{K01, NZ17}\footnote{Notice that~\eqref{eq:fLPPint} and~\eqref{eq:hLPPint} have now been obtained directly via the standard Robinson-Schensted-Knuth correspondence, avoiding the route of the zero temperature limit, see \cite{Bis18, BZ19b}.}.
%Although the derivation of the Laplace transforms for the log-gamma polymer in~\cite{BZ19a} used combinatorial arguments through the geometric Robinson-Schensted-Knuth correspondence~\cite{K01, NZ17},
%it is not obvious how to obtain relations~\eqref{eq:fLPPint} and~\eqref{eq:hLPPint} directly by a combinatorial argument such as the classical Robinson-Schensted-Knuth correspondence,
%avoiding the route of the zero temperature limit.
Continuous classical and symplectic Schur functions appear then in our context as scaling limits of Whittaker functions associated to the groups $\GL_N(\R)$ and $\SO_{2N+1}(\R)$ respectively.
\vskip 2mm
{\bf Organization of the paper}: In Section~\ref{sec:fredholm} we present a general scheme to turn a ratio of determinants to a Fredholm determinant. In Section~\ref{sec:steepestDescent} we perform the steepest descent analysis of a central integral. In Sections~\ref{sec:fLPP} and~\ref{sec:hLPP} we prove Theorems~\ref{thm:fAsympt} and~\ref{thm:hAsympt} respectively.

\section{From determinants to Fredholm determinants}
\label{sec:fredholm}

In this section we present a general scheme to turn ratios of determinants into a Fredholm determinant. Such a scheme is not new (see for example \cite{O01, J01, BG12}), but we present it here in a fashion adapted to our framework. Let us start by briefly recalling the notion of a Fredholm determinant: Given a measure space $(\mathcal{X},\mu)$, any linear operator $K\colon L^2(\mathcal{X}) \to L^2(\mathcal{X})$ can be defined in terms of its integral kernel $K(x,y)$ by
\[
(Kh)(x) := \int_{\mathcal{X}} K(x,y)h(y) \mu(\diff y) \, ,\qquad h\in L^2(\mathcal{X}) \, .
\]
The \emph{Fredholm determinant} of $K$ can then be defined through its series expansion:
\begin{equation}
\label{eq:fredholmDet}
\det(I + K)_{L^2(\mathcal{X})}
:= 1+\sum_{n=1}^{\infty} \frac{1}{n!} \int_{\mathcal{X}^n}
\det(K(x_i,x_j))_{i,j=1}^n \,
\mu(\diff x_1)\dots \mu(\diff x_n) \, ,
\end{equation}
assuming the series converge. Denoting from now on by $\C_+ := \{z\in\C\colon \Re(z)>0\}$ the complex right half-plane, we can state:
\begin{theorem}
\label{thm:detToFredDet}
Let
\begin{equation}
\label{eq:Hbar}
H(z,w) := C(z,w) - \widehat{H}(z,w) \, ,
\end{equation}
where $C$ is the function
\begin{align*}
C(z,w)=\frac{1}{z+w},\qquad \text{for } (z,w) \in \mathbb{C}^2 \, ,
\end{align*}
 and $\widehat{H}$ is a holomorphic function in the region $\C_+ \times \C_+$. For any choice of positive parameters $\alpha_1,\dots,\alpha_N,\beta_1,\dots,\beta_N$, define the operator $K_N$ on $L^2(\R_+)$ through the kernel
\begin{equation}
\label{eq:kernel}
K_N(\lambda,\xi) := \frac{1}{(2\pi\i)^2} \int_{\Gamma_1} \diff z \int_{\Gamma_2} \diff w \,
\e^{-\lambda z-\xi w}
\widehat{H}(z,w) \,
\prod_{m=1}^N \bigg[ \frac{(z+\beta_m)(w+\alpha_m)}{(z-\alpha_m)(w-\beta_m)} \bigg] \, ,
\end{equation}
where $\Gamma_1,\Gamma_2 \subset \C_+$ are any positively oriented simple closed contours such that $\Gamma_1$ encloses $\alpha_1,\dots,\alpha_N$ and $\Gamma_2$ encloses $\beta_1,\dots,\beta_N$. Then
\begin{equation}
\label{eq:detToFredDet}
\frac{\det(H(\alpha_i,\beta_j))_{1\leq i,j\leq N}}
{\det(C(\alpha_i,\beta_j))_{1\leq i,j\leq N}}
= \det(I - K_N)_{L^2(\R_+)} \, .
\end{equation}
\end{theorem}
Note that the denominator on the left hand side of~\eqref{eq:detToFredDet} is a Cauchy determinant:
\begin{equation}
\label{eq:CauchyDet}
\det(C(\alpha_i,\beta_j))_{i,j}
= \det\bigg(\frac{1}{\alpha_i + \beta_j}\bigg)_{i,j}
= \frac{\prod_{1\leq i<j\leq N} (\alpha_i -\alpha_j)(\beta_i - \beta_j)}
{\prod_{1\leq i,j\leq N} (\alpha_i + \beta_j)} \, .
\end{equation}
\begin{proof}
For convenience, let us denote by $\mathcal{C}$, $\mathcal{H}$ and $\mathcal{\widehat{H}}$ the $N\times N$ matrices $(C(\alpha_i,\beta_j))_{1\leq i,j\leq N}$, $(H(\alpha_i,\beta_j))_{1\leq i,j\leq N}$ and $(\widehat{H}(\alpha_i,\beta_j))_{1\leq i,j\leq N}$ respectively. We then have:
\begin{equation}
\label{eq:detToFredDet1}
\frac{\det(\mathcal{H})}{\det(\mathcal{C})}
= \det\big( \mathcal{C}^{-1}\big( \mathcal{C} - \mathcal{\widehat{H}} \big) \big)
= \det\big( I - \mathcal{C}^{-1} \mathcal{\widehat{H}} \big) \, ,
\end{equation}
where $I$ is the identity matrix of order $N$. To invert $\mathcal{C}$, we use Cramer's formula:
\[
\mathcal{C}^{-1}(i,k) = (-1)^{i+k} \frac{\det\big(\mathcal{C}^{(k,i)}\big)}{\det(\mathcal{C})} \, ,
\]
where $\mathcal{C}^{(k,i)}$ is the matrix of order $N-1$ obtained from $\mathcal{C}$ by removing its $k$-th row and $i$-th column. In our case, both determinants in the above formula are of Cauchy type:
\begin{align*}
\det(\mathcal{C})
&= \prod_{l<m} (\alpha_l -\alpha_m)\prod_{l<m} (\beta_l - \beta_m)
\prod_{l,m} (\alpha_l + \beta_m)^{-1} \\
\det\big(\mathcal{C}^{(k,i)}\big)
&= \prod_{\substack{l<m \\ l,m\neq k}} (\alpha_l -\alpha_m)
\prod_{\substack{l<m \\ l,m\neq i}}(\beta_l - \beta_m)
\prod_{\substack{l\neq k \\ m\neq i}} (\alpha_l + \beta_m)^{-1} \, ,
\end{align*}
where indices $l$ and $m$ run in $\{1,\dots,N\}$.
The inverse of $\mathcal{C}$ is thus readily computed: 
\[
\mathcal{C}^{-1}(i,k)
= \frac{\prod_{m=1}^N
(\alpha_k + \beta_m)
(\beta_i + \alpha_m)}
{(\alpha_k + \beta_i)
\prod_{m\neq k}(\alpha_k-\alpha_m)
\prod_{m\neq i} (\beta_i - \beta_m)} \, .
\]
Writing $(\alpha_k + \beta_i)^{-1} = \int_0^{\infty} \exp(-(\alpha_k + \beta_i)\lambda ) \diff \lambda $, we obtain:
\[
\big(\mathcal{C}^{-1} \mathcal{\widehat{H}} \big)(i,j)
= \sum_{k=1}^N \mathcal{C}^{-1}(i,k) \mathcal{\widehat{H}}(k,j)
= \int_0^{\infty} f(i,\lambda ) g(\lambda ,j) \diff \lambda  \, ,
\]
where for all $\lambda >0$
\[
f(i,\lambda ) := \e^{-\beta_i \lambda } \frac{\prod_{m=1}^N(\beta_i + \alpha_m)}
{\prod_{m\neq i}(\beta_i -\beta_m)} \, ,
\qquad
g(\lambda ,j) := \sum_{k=1}^N \e^{-\alpha_k \lambda } \frac{\prod_{m=1}^N (\alpha_k + \beta_m)}
{\prod_{m\neq k} (\alpha_k - \alpha_m)}
\widehat{H}(\alpha_k,\beta_j) \, .
\]
This proves that matrix $\mathcal{C}^{-1} \mathcal{\widehat{H}}$, viewed as a linear operator on $\R^N$, equals the composition $FG$, where $F$ and $G$ are the linear operators
\begin{align*}
&F \colon L^2(\R_+) \to \R^N \, ,
&&\phi \mapsto \bigg[ \int_0^{\infty} f(i,\lambda ) \phi(\lambda ) \diff \lambda  \bigg]_{i=1}^N \, , \\
&G \colon \R^N \to L^2(\R_+) \, ,
&&(a(j))_{j=1}^N \mapsto \sum_{j=1}^N g(\lambda ,j) a(j) \, .
\end{align*}
We note that these are well-defined operators, as $f(i, \cdot)$ and $g(\cdot,j)$ are square integrable functions on $\R_+$, for all $i$ and $j$. 
We will next use Sylvester's identity, which states that
\begin{equation}
\label{eq:sylvester}
\det(I + K_1 K_2)_{L^2(\mathcal{X}_2)}
= \det(I + K_2 K_1)_{L^2(\mathcal{X}_1)}
\end{equation}
for any trace class operators $K_1\colon L^2(\mathcal{X}_1) \to L^2(\mathcal{X}_2)$ and $K_2\colon L^2(\mathcal{X}_2) \to L^2(\mathcal{X}_1)$.
By applying this identity, we obtain
\begin{equation}
\label{eq:detToFredDet2}
\det\big( I - \mathcal{C}^{-1} \mathcal{\widehat{H}} \big)_{\R^N}
= \det(I - F G)_{\R^N}
= \det(I - K_N)_{L^2(\R_+)} \, ,
\end{equation}
where $K_N:= G F$ is the operator on $L^2(\R_+)$ defined through the kernel
\begin{equation}
\label{eq:kernelDoubleSum}
K_N(\lambda ,\xi ) :=
\sum_{i,k=1}^N \e^{-\alpha_k \lambda  - \beta_i \xi }
\widehat{H}(\alpha_k, \beta_i)
\frac{\prod_{m=1}^N
(\alpha_k + \beta_m)
(\beta_i + \alpha_m)}
{\prod_{m\neq k}(\alpha_k-\alpha_m)
\prod_{m\neq i} (\beta_i - \beta_m)} \, .
\end{equation}
From the latter formula, it is clear that $|K_N(\lambda,\xi)| \leq c_1 \e^{- c_2 \lambda}$ for all $\lambda\in [0,\infty)$, where the positive constants $c_1$ and $c_2$ depend on $N$ and on the parameters. Hadamard's bound then implies that
\[
\abs{ \det(K_N(\lambda_i, \lambda_j))_{i,j=1}^n }
\leq n^{n/2} \prod_{i=1}^n c_1 \e^{-c_2 \lambda_i } \, .
\]
It follows that
\[
\begin{split}
\abs{ \det(I - K_N)_{L^2(\R_+)} }
\leq 1 + \sum_{n=1}^{\infty} \frac{n^{n/2}}{n!} \bigg( \int_{0}^{\infty} c_1 \e^{-c_2 \lambda} \diff \lambda \bigg)^n
< \infty \, ,
\end{split}
\]
hence the right-hand side of~\eqref{eq:detToFredDet2} is indeed a converging Fredholm determinant.
By applying the residue theorem (recalling the assumption that $\widehat{H}(z,w)$ is holomorphic in $\C_+ \times \C_+$), the double sum in~\eqref{eq:kernelDoubleSum} can be turned into a double contour integral, yielding representation~\eqref{eq:kernel} for the kernel. By combining~\eqref{eq:detToFredDet1} and~\eqref{eq:detToFredDet2} we obtain~\eqref{eq:detToFredDet}.
\end{proof}

\section{Steepest descent analysis}
\label{sec:steepestDescent}

Thanks to the results mentioned in section~\ref{sec:model} (formulae~\eqref{eq:fLPPint} and~\eqref{eq:hLPPint}) and to the general scheme introduced in section~\ref{sec:fredholm}, the distribution function of our two dLPP models can be expressed as a Fredholm determinant on $L^2(\R_+)$ with kernel of type~\eqref{eq:kernel}, as we will state precisely in the next two sections (see Theorems~\ref{thm:fFredholm} and~\ref{thm:hFredholm}).
As we will see, in the limit $N\to\infty$ all these kernels converge, after rescaling, to expressions involving Airy functions.
In order to see this, one needs to perform the asymptotic analysis of a few contour integrals via steepest descent.
This procedure is very similar in all cases, as it always involves the same functions. Therefore, we will carry it out in detail only for one of such contour integrals, arguably the most archetypal one as it just approximates the Airy function. Other very similar steepest descent analyses are sketched where needed, specifically in the proof of Theorem~\ref{thm:hAsympt}.

Let us first recall that the Airy function $\Ai$ has the following contour integral representation:
\begin{equation}
\label{eq:Airy}
\Ai(x) := \frac{1}{2\pi \i}
\int_{\e^{-\i\pi/3} \infty}^{\e^{\i\pi/3} \infty}
\exp\bigg\{\frac{z^3}{3} -xz \bigg\} \diff z \, ,
\end{equation}
where the integration path starts at infinity with argument $-\pi/3$ and ends at infinity with argument $\pi/3$ (see for example the red contour in figure~\ref{fig:steepestDescentPath}).

\begin{proposition}
\label{prop:steepestDescent}
For any fixed $\gamma>0$ and $f:=2/\gamma$, let us define
\begin{equation}
\label{eq:steepestDescentInt}
J_N(x) := -\frac{1}{2\pi\i} \int_{\Gamma}
\e^{-z(fN + x)}
\bigg[\frac{\gamma + z}{\gamma - z}\bigg]^N \diff z \, ,
\end{equation}
where $\Gamma\subset \C$ is a positively oriented contour enclosing $\gamma$. Then, for all $x\in\R$,
\begin{equation}
\label{eq:convergence2Airy}
\tilde{J}_N(x)
:= \frac{\sqrt[3]{2N}}{\gamma} J_N\bigg( \frac{\sqrt[3]{2N}}{\gamma} x \bigg)
\xrightarrow{N\to\infty} \Ai(x) \, .
\end{equation}
\end{proposition}

\begin{proof}
The proof is based on the steepest descent analysis of the integral
\[
J_N(x)
= - \frac{1}{2\pi\i} \int_{\Gamma}
\exp\big\{ N F(z) - x z \big\} \diff z \, ,
\]
where $F(z):= \log(\gamma+z) - \log(\gamma - z) - fz$. 
We need to compute the critical points of the function $F$, whose first three derivatives are given by:
\begin{align*}
F'(z) &= \frac{1}{\gamma+z} + \frac{1}{\gamma-z} - f \, , \\
F''(z) &= -\frac{1}{(\gamma+z)^2} + \frac{1}{(\gamma-z)^2} \, , \\
F'''(z) &= \frac{2}{(\gamma+z)^3} + \frac{2}{(\gamma-z)^3} \, .
\end{align*}
The second derivative vanishes if and only if $z=0$. As in the statement of the theorem, we then set $f:=2/\gamma$ , which is the only value of $f$ such that the first derivative also vanishes at $z=0$. The first non-vanishing derivative of $F$ at the critical point $z=0$ is then the third one. In particular, we have that
\[
F(0)=F'(0)=F''(0)=0 \, , \qquad
F'''(0)=\frac{4}{\gamma^3} \, ,
\]
hence the Taylor expansion of $F$ near the critical point is
\begin{equation}
\label{eq:expansionCriticalPoint}
F(z) = \frac{2}{3\gamma^3} z^3 + R(z) \, ,
\end{equation}
where $R(z) = o(z^3)$ as $z\to 0$.
\begin{figure}
\begin{tikzpicture}[scale=1.5]
\draw (0,-3.2) -- (0,3.2);
\draw (0,0) -- (2.2,0);

\node[label={[label distance=1pt]180:$0$},draw,circle,inner sep=1pt,fill] (origin) at (0,0) {};

\node[label={[label distance=1pt]0:$2a\e^{\i\pi/3}$},draw,circle,inner sep=1pt,fill] (V+) at (1,1.73) {};

\node[label={[label distance=1pt]0:$2a\e^{-\i\pi/3}$},draw,circle,inner sep=1pt,fill] (V-) at (1,-1.73) {};

\node[label={[label distance=1pt]315:$a$},draw,circle,inner sep=1pt,fill] (a) at (1,0) {};

\node[label={[label distance=1pt]270:$\gamma$},draw,circle,inner sep=1pt,fill] (gamma) at (0.7,0) {};

\node[label={[label distance=5pt]0:$T_a$}] (gamma) at (1,0.8) {};

\node[inner sep=0pt, label={[label distance=5pt]0:\red{$C$}}] (W+) at (1.515,2.595) {};
\node[inner sep=0pt] (W++) at (1.715,2.941) {};
\node[inner sep=0pt] (W-) at (1.515,-2.595) {};
\node[inner sep=0pt] (W--) at (1.715,-2.941) {};
%\node[label={[label distance=1pt]5:\red{$A_{\delta}$}}] at (0.02,0) {};

\begin{scope}[thick,decoration={
    markings,
    mark=at position 0.5 with {\arrow[scale=2.5]{stealth}}}
    ] 
\draw[postaction={decorate}] (origin) -- (V+);
\draw[postaction={decorate}] (V-) -- (origin);
\end{scope}

\begin{scope}[thick,decoration={
    markings,
    mark=at position 0.35 with {\arrow[scale=2.5]{stealth}}}
    ] 
\draw[postaction={decorate}] (V+) -- (V-);
\end{scope}
    
\begin{scope}[thick,red]
\draw[postaction={decorate}, decoration={markings, mark=at position 0.15 with {\arrow{>}}}]
(W-) -- (0.015,0);
\draw[postaction={decorate}, decoration={markings, mark=at position 0.85 with {\arrow{>}}}]
(0.015,0) -- (W+);
\draw[dashed] (W+) -- (W++);
\draw[dashed] (W-) -- (W--);
\end{scope}

\end{tikzpicture}
\caption{The red path $C$ is involved in the integral representation of the Airy function. The black contour $T_a$ refers to the steepest descent analysis in the proof of Proposition~\ref{prop:steepestDescent}.}
\label{fig:steepestDescentPath}
\end{figure}
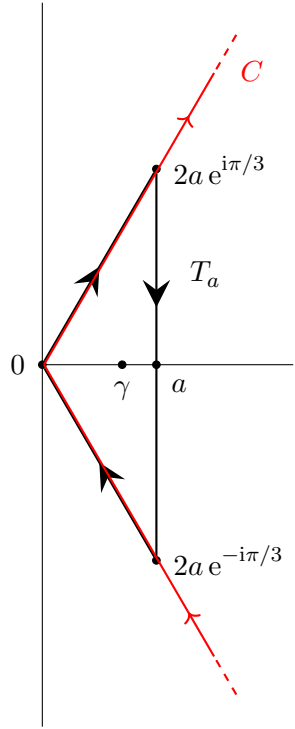
Since the directions of steepest descent of $F$ from $z=0$ correspond to the angles $\pm \pi/3$, we deform the \emph{positively} oriented contour $\Gamma$ into the \emph{negatively} oriented triangular path $T_a$ with vertices $0$, $2a\e^{\i\pi/3}$ and $2a\e^{-\i\pi/3}$ for some $a>\gamma$ (so that the pole $z=\gamma$ is still enclosed, see figure~\ref{fig:steepestDescentPath}). This only implies a change of sign in the integral, corresponding to the change of orientation of the contour. Indeed, in order to obtain the right estimates in the proof of Corollary~\ref{coro:steepestDescentEstimate}, it is convenient to consider an infinitesimal shift of $T_a$, by setting the contour to be $T_a + \epsilon\gamma/\sqrt[3]{2N}$, where $\epsilon > 0$ is an arbitrary constant. Moreover, we split the integral into two regions, i.e.\ a neighborhood of the critical point, where the main contribution of the integral is expected to come from, and its exterior (we choose the neighborhood to be a ball centered at $\epsilon\gamma / \sqrt[3]{2N}$ with radius $\gamma N^{-\alpha}$, where $\alpha>0$ will be suitably specified later on):
\[
J_N(x)
= J_N^{\rm in}(x) + J_N^{\rm ex}(x) \, ,
\]
where
\begin{align*}
J_N^{\rm in}(x)
&:= \frac{1}{2\pi\i} \int_{T_a + \frac{\epsilon\gamma}{\sqrt[3]{2N}}}
\exp\big\{ F(z) N - x z \big\}
\1_{\{\abs{z - \epsilon\gamma/\sqrt[3]{2N}} \leq \gamma N^{-\alpha}\}} \diff z \, , \\
J_N^{\rm ex}(x)
&:= \frac{1}{2\pi\i} \int_{T_a + \frac{\epsilon\gamma}{\sqrt[3]{2N}}}
\exp\big\{ F(z) N - x z \big\}
\1_{\{\abs{z - \epsilon\gamma/\sqrt[3]{2N}} \geq \gamma N^{-\alpha}\}} \diff z \, .
\end{align*}

Let us first focus on the former integral and denote by $C$ the piecewise linear path going from the point at infinity with argument $-\pi/3$ to the origin to the point at infinity with argument $\pi/3$ (see figure~\ref{fig:steepestDescentPath}). We then have that
\[
J_N^{\rm in}(x)
= \frac{1}{2\pi\i} \int_{C + \frac{\epsilon\gamma}{\sqrt[3]{2N}}}
\exp\bigg\{ \frac{2N}{3\gamma^3} z^3 - x z + R(z) N \bigg\}
\1_{\{\abs{z - \epsilon\gamma/\sqrt[3]{2N}} \leq \gamma N^{-\alpha}\}} \diff z \, ,
\]
where $R(z)$ is defined by~\eqref{eq:expansionCriticalPoint}.
If we now rescale both the integration variable and the function  $J_N^{\rm in}$ by the factor $\sqrt[3]{2N}/\gamma$, by setting $w := z \sqrt[3]{2N}/\gamma$ and defining $\tilde{J}_N^{\rm in}$ as in~\eqref{eq:convergence2Airy}, we obtain:
\begin{equation}
\label{eq:steepestDescentIn}
\tilde{J}_N^{\rm in}(x)
= \frac{1}{2\pi\i} \int_{C + \epsilon}
\exp\bigg\{ \frac{w^3}{3} - x w
+  R\bigg( \frac{\gamma}{\sqrt[3]{2N}} w \bigg) N \bigg\}
\1_{\{\abs{w - \epsilon} \leq \sqrt[3]{2} N^{1/3-\alpha}\}}
\diff w \, .
\end{equation}
A standard estimate of the remainder in the Taylor expansion~\eqref{eq:expansionCriticalPoint} yields
\[
\abs{R\bigg( \frac{\gamma}{\sqrt[3]{2N}} w \bigg) N} 
\leq \frac{m}{4!} \abs{\frac{\gamma}{\sqrt[3]{2N}} w}^4 N
\leq \frac{m}{4!} \bigg(\gamma N^{-\alpha} + \frac{\epsilon\gamma}{\sqrt[3]{2N}} \bigg)^4 N
\]
for $\abs{w - \epsilon} \leq \sqrt[3]{2} N^{1/3-\alpha}$, where the constant $m$ is the maximum modulus of $F^{(4)}$ in some fixed neighborhood of the origin. If we take $\alpha > 1/4$, the above expression vanishes as $N\to\infty$. If we further choose $\alpha < 1/3$, the indicator function in~\eqref{eq:steepestDescentIn} converges to $1$, yielding
\begin{equation}
\label{eq:steepestDescentInConv}
\exp\bigg\{ R_N\bigg( \frac{\gamma}{\sqrt[3]{2N}} w \bigg) N \bigg\} \,
\1_{\{\abs{w - \epsilon} \leq \sqrt[3]{2} N^{1/3-\alpha}\}}
\xrightarrow{N\to\infty} 1 \, .
\end{equation}
Since the argument of the points of $C$ is $\pm \pi/3$, we have that
\[
\int_{C + \epsilon}
\abs{\exp\bigg\{\frac{w^3}{3} -xw \bigg\} } \abs{\diff w}
< \infty \, ,
\]
hence by dominated convergence
\[
\tilde{J}_N^{\rm in}(x)
\xrightarrow{N\to\infty} \int_{C + \epsilon}
\exp\bigg\{\frac{w^3}{3} -xw \bigg\} \diff w
= \Ai(x) \, .
\]
Observe that, varying $\epsilon$, we have different integral representations of the Airy function, which are all equivalent thanks to~\eqref{eq:Airy}.

To conclude the proof, it remains to show that
\[
\tilde{J}_N^{\rm ex}(x)
:= \frac{\sqrt[3]{2N}}{2\pi\i\gamma} \int_{T_a \cap \{\abs{z} \geq \gamma N^{-\alpha}\}}
\exp\bigg\{ F\bigg(z + \frac{\epsilon\gamma}{\sqrt[3]{2N}} \bigg) N - x \bigg(z + \frac{\epsilon\gamma}{\sqrt[3]{2N}} \bigg) \frac{\sqrt[3]{2N}}{\gamma} \bigg\}
\diff z
\xrightarrow{N\to\infty} 0 \, .
\]
We may decompose the integration domain as
\[
T_a \cap \{\abs{z} \geq \gamma N^{-\alpha}\}
= \mathcal{V}
\cup \mathcal{O}_N
\cup \overline{\mathcal{V}}
\cup \overline{\mathcal{O}}_N \, ,
\]
where $\mathcal{V}$ and $\mathcal{O}_N$ are the vertical and oblique segments respectively given by
\[
\mathcal{V} := \Big\{\Re(z)=a \, , \,\, 0 \leq \arg(z) \leq \frac{\pi}{3} \Big\} \, ,
\quad
\mathcal{O}_N := \Big\{ \gamma N^{-\alpha} \leq \abs{z} \leq 2a \, , \,\, \arg(z) = \frac{\pi}{3} \Big\} \, ,
\]
and $\overline{\mathcal{V}}$ and $\overline{\mathcal{O}}_N$ are their complex conjugates.
We thus estimate
\begin{equation}
\label{eq:steepestDescentExt}
\abs{\tilde{J}_N^{\rm ex}(x)}
\leq \frac{\mathcal{L}(T_a) \sqrt[3]{2N}}{2\pi\gamma} \exp\bigg\{ \max \bigg[ G_N\bigg(z + \frac{\epsilon\gamma}{\sqrt[3]{2N}} \bigg) \colon z\in \mathcal{V} \cup \mathcal{O}_N \cup \overline{\mathcal{V}} \cup \overline{\mathcal{O}}_N \bigg] \bigg\} \, ,
\end{equation}
where $\mathcal{L}(\cdot)$ denotes the length of a contour and
\[
G_N(z) := \Re[F(z)] N - x \, \Re(z) \frac{\sqrt[3]{2N}}{\gamma} \, .
\]
Since
\[
\Re[F(z)] = \log\abs{\frac{\gamma+z}{\gamma-z}} - \frac{2}{\gamma} \Re(z) \, ,
\]
it is clear that $G_N(\overline{z}) = G_N(z)$. Therefore, it suffices to bound the maximum of $G_N$ over $\mathcal{V}$ and over $\mathcal{O}_N$.
Since $\Re(z)=a$ and $a\leq \abs{z}\leq 2a$ for $z\in\mathcal{V}$, we have that
\begin{equation}
\label{eq:steepestDescentExtVert}
\max_{z\in\mathcal{V}} G_N\bigg(z + \frac{\epsilon\gamma}{\sqrt[3]{2N}} \bigg)
\leq - c N - x a \frac{\sqrt[3]{2N}}{\gamma} \, ,
\end{equation}
where
\begin{equation}
\label{eq:steepestDescentExtVertConst}
c := \frac{2}{\gamma} a - \log\bigg(\frac{2a + \epsilon\gamma +\gamma}{a -\gamma}\bigg) \, .
\end{equation}
If we fix a large enough $a$ such that $c$ is positive, the maximum in~\eqref{eq:steepestDescentExtVert} is asymptotic to $-cN$ and diverges to $-\infty$.
On the other hand, for all $z$ such that $\arg(z)=\pi/3$, we have that
\[
G_N(z)
= \bigg[ \frac{1}{2} \log\frac{\gamma^2 + \gamma\abs{z} + \abs{z}^2}
{\gamma^2 - \gamma\abs{z} + \abs{z}^2}
- \frac{\abs{z}}{\gamma} \bigg] N
- x \frac{\abs{z}}{2} \frac{\sqrt[3]{2N}}{\gamma} \, ,
\]
whose derivative w.r.t.\ the modulus is
\[
\frac{\diff}{\diff \abs{z}} G_N(z)
= -\frac{\abs{z}^2(2\gamma^2+\abs{z}^2)}
{\gamma(\gamma^4+\gamma^2\abs{z}^2+\abs{z}^4)} N
- \frac{x \sqrt[3]{2N}}{2\gamma} \, .
\]
A trivial estimate then gives
\[
\frac{\diff}{\diff \abs{z}} G_N(z)
\leq - \frac{\gamma N^{1-2\alpha} (2\gamma^2 + \gamma^2 N^{-2\alpha})}
{(\gamma^4 + \gamma^2 (2a)^2 + (2a)^4)}
- \frac{x \sqrt[3]{2N}}{2\gamma}
\qquad
\text{for } z\in\mathcal{O}_N \, .
\]
Since $\alpha < 1/3$, no matter the sign of $x$, the above derivative is negative for $N$ large enough, so $G_N(z)$ is decreasing w.r.t.\ $\abs{z}$ in $\mathcal{O}_N$. By continuity, for $N$ large, $G_N(z + \epsilon\gamma/\sqrt[3]{2N})$ is also decreasing w.r.t.\ $\abs{z}$ 
in $\mathcal{O}_N$, hence
\[
\begin{split}
&\max_{z\in \mathcal{O}_N} G_N\bigg(z + \frac{\epsilon\gamma}{\sqrt[3]{2N}} \bigg)
= G_N\bigg(\gamma N^{-\alpha} \e^{\i\pi/3} + \frac{\epsilon\gamma}{\sqrt[3]{2N}} \bigg) \\
= \, &\Bigg[ \log\Bigg| \frac{1 + N^{-\alpha} \e^{\i\pi/3} + \epsilon(2N)^{-1/3}}
{1 - N^{-\alpha} \e^{\i\pi/3} - \epsilon(2N)^{-1/3}} \Bigg|
- N^{-\alpha} -\frac{2\epsilon}{\sqrt[3]{2N}} \Bigg] N
- x\bigg(\frac{N^{-\alpha}}{2} + \frac{\epsilon}{\sqrt[3]{2N}} \bigg) \sqrt[3]{2N} \, .
\end{split}
\]
After a tedious computation, which uses the third order Taylor expansion of $\log(1+\delta)$ as $\delta \to 0$, we obtain that
\begin{equation}
\label{eq:steepestDescentExtObl}
\max_{z\in \mathcal{O}_N} G_N\bigg(z + \frac{\epsilon\gamma}{\sqrt[3]{2N}} \bigg)
= -\frac{2}{3} N^{1-3\alpha} + o(N^{1-3\alpha})
- x (2^{-2/3} N^{1/3-\alpha}
+ \epsilon) \, .
\end{equation}
Since $\alpha < 1/3$, the latter expression is asymptotic to $-(2/3)N^{1-3\alpha}$ and diverges to $-\infty$. Thanks to estimates~\eqref{eq:steepestDescentExt}, \eqref{eq:steepestDescentExtVert} and~\eqref{eq:steepestDescentExtObl}, we thus conclude that $\tilde{J}_N^{\rm ex}(x)$ vanishes (at least choosing a large enough $a$).
\end{proof}

The proof of Proposition~\ref{prop:steepestDescent} directly provides a uniform bound on $\tilde{J}_N$, which will also turn out to be useful in the next sections.

\begin{corollary}
\label{coro:steepestDescentEstimate}
Let $\tilde{J}_N(x)$ be defined as in~\eqref{eq:convergence2Airy} and $s\in\R$. Then, there exist two positive constants $c_1$ and $c_2$ such that
\[
\sup_{N\in\N} \big| \tilde{J}_N(x) \big| \leq c_1 \e^{-c_2 x} \quad\qquad
\forall x\in [s,\infty) \, .
\]
\end{corollary}
\begin{proof}
Since by continuity the converging sequence $\tilde{J}_N(x)$ is bounded uniformly in $N$ on any compact set, it suffices to prove the claim for $s=0$.
The proof is then a straightforward consequence of the estimates obtained in the proof of Proposition~\ref{prop:steepestDescent}.
Using the notation adopted there, we will show that the uniform exponential bound is valid for both $\tilde{J}^{\rm in}_N$ and $\tilde{J}^{\rm ex}_N$, i.e.\ the contributions near and away from the critical point respectively.
From~\eqref{eq:steepestDescentIn} and~\eqref{eq:steepestDescentInConv}, it follows that for all $x \in [0,\infty)$
\[
\sup_{N\in\N} \big| \tilde{J}^{\rm in}_N(x) \big|
\lesssim \e^{-\epsilon x}
\int_{C+\epsilon} \e^{\Re(w^3)/3} \abs{\diff w}
\, ,
\]
with $\epsilon$ chosen to be strictly positive.
By definition of the contour $C$, the above integral converges, providing the desired exponential bound for $\tilde{J}^{\rm in}_N$. On the other hand, estimates~\eqref{eq:steepestDescentExt}, \eqref{eq:steepestDescentExtVert} and~\eqref{eq:steepestDescentExtObl} show that for all $N\in\N$ and $x\in [0,\infty)$
\[
\begin{split}
\big| \tilde{J}^{\rm ex}_N(x) \big|
&\leq
\bigg[ \frac{\mathcal{L}(T_a) \sqrt[3]{2N}}{2\pi\gamma} \e^{ - \min\{ cN, (2/3) N^{1-3\alpha} + o(N^{1-3\alpha}) \} } \bigg]
\e^{-x \min\{ a\sqrt[3]{2N}/\gamma, 2^{-2/3} N^{1/3-\alpha} + \epsilon \}} \\
&\leq c' \e^{-x \min\{ a\sqrt[3]{2}/\gamma, 2^{-2/3} + \epsilon \}} \, ,
\end{split}
\]
where $c$ is the constant (positive if $a$ is chosen large enough) defined in~\eqref{eq:steepestDescentExtVertConst}, and $c'$ is an upper bound for the vanishing sequence inside the square bracket above.
This provides the desired exponential bound for $\tilde{J}^{\rm ex}_N$.
\end{proof}

\section{Point-to-line last passage percolation and GOE}
\label{sec:fLPP}

We will now specialize the results of sections~\ref{sec:fredholm} and~\ref{sec:steepestDescent} to the models described in section~\ref{sec:model}. We first analyze the point-to-line dLPP, writing its distribution function as a Fredholm determinant. The starting point is the next theorem, which is a result of \eqref{eq:contSp} and \eqref{eq:fLPPint}, combined with the Cauchy-Binet identity.
\begin{theorem}[\cite{BZ19a}]
\label{thm:fLPPdet}
The distribution function of the point-to-line directed last passage percolation $\fTau_{2N}$ with exponential waiting times as in~\eqref{eq:expDistribution} is a ratio of $N\times N$ determinants:
\begin{align}
\label{eq:fLPPdet}
\P\big(\fTau_{2N}\leq u\big) &= 
\frac{\det\big(\fH_u(\alpha_i,\beta_j)\big)_{1\leq i,j\leq N}}{\det(C(\alpha_i,\beta_j))_{1\leq i,j\leq N}}
\end{align}
for $u>0$, where $C(z,w) := (z+w)^{-1}$ and
\begin{align*}
\fH_u(z,w) &:= \e^{-u(z + w)} \int_0^u (\e^{z x}-\e^{-z x})
(\e^{w x}-\e^{-w x}) \diff x \, .
\end{align*}
\end{theorem}

We next obtain the Fredholm determinant.
\begin{theorem}
\label{thm:fFredholm}
The distribution of $\fTau_{2N}$ with exponential waiting times as in~\eqref{eq:expDistribution} is given by
\begin{equation}
\label{eq:fFredholm}
\P\big(\fTau_{2N}\leq u\big) = \det(I - \fK_{N,u})_{L^2(\R_+)} \, ,
\end{equation}
where $\fK_{N,u} \colon L^2(\R_+)\to L^2(\R_+)$ is the operator defined through the kernel
\begin{equation}
\label{eq:fKernel}
\fK_{N,u}(\lambda,\xi)
= \frac{1}{(2\pi\i)^2}
\int_{\Gamma_1} \diff z
\int_{\Gamma_2} \diff w \,
\e^{-\lambda z - \xi w} \,
\fHbar_u(z,w)
\prod_{m=1}^N \frac{(z+\beta_m)(w+\alpha_m)}{(z-\alpha_m)(w-\beta_m)} \, .
\end{equation}
Here, $\Gamma_1,\Gamma_2 \subset \C_+$ are any positively oriented simple closed contours such that $\Gamma_1$ encloses $\alpha_1,\dots,\alpha_N$, $\Gamma_2$ encloses $\beta_1,\dots,\beta_N$ as well as the whole $\Gamma_1$, and
\begin{equation}
\label{eq:fHbar}
\fHbar_u(z,w)
= \frac{\e^{-2uz}}{w-z}
+ \frac{\e^{-2uw}}{z-w}
+ \frac{\e^{-2u(z+w)}}{z+w} \, .
\end{equation}
\end{theorem}

\begin{proof}
The claim is an immediate consequence of Theorem~\ref{thm:fLPPdet} (see formula~\eqref{eq:fLPPdet}) and Theorem~\ref{thm:detToFredDet}. According to~\eqref{eq:Hbar}, function $\fHbar_u$ in~\eqref{eq:fKernel} is defined through the relation $\fH_u = C - \fHbar_u$ (using the notation of Theorem~\ref{thm:fLPPdet} for $\fH_u$),  i.e.
\[
\fHbar_u(z,w)
= \frac{1}{z+w} - \e^{-u(z + w)} \int_0^u (\e^{z x}-\e^{-z x})
(\e^{w x}-\e^{-w x}) \diff x \, .
\]
If we assume that $\Gamma_2$ encloses $\Gamma_1$ (so that $z\neq w$ for all $z,w$), integrating the above expression yields~\eqref{eq:fHbar}. By symmetry, the other inclusion would lead to similar results.
\end{proof}

\begin{proof}[{\bf Proof of Theorem~\ref{thm:fAsympt}}]
The starting point is the Fredholm determinant formula of Theorem~\ref{thm:fFredholm}. 
We will first show the pointwise convergence of the kernel after suitable rescaling, and next sketch the (standard) argument for the convergence of the Fredholm determinant.
Setting $\alpha_m=\beta_m=\gamma$ for all $m$, so that the waiting times are all exponential with rate $2\gamma$ (see parametrization~\eqref{eq:expDistribution}), kernel~\eqref{eq:fKernel} reads as
\[
\fK_{N,u}(\lambda,\xi)
= \frac{1}{(2\pi\i)^2}
\int_{\Gamma_1} \diff z
\int_{\Gamma_2} \diff w \,
\e^{-\lambda z - \xi w} \,
\fHbar_u(z,w)
\bigg[ \frac{(z+\gamma)(w+\gamma)}{(z-\gamma)(w-\gamma)} \bigg]^N \, ,
\]
where $\fHbar_u$ is given by formula~\eqref{eq:fHbar}.
Our kernel is thus a sum of three double contour integrals, each corresponding to one of the addends in~\eqref{eq:fHbar}. In the second one \emph{only}, we swap the two contours taking into account the residue at the pole $w=z$. We can then readily write the kernel as the sum of four terms:
\[
\fK_{N,u}
= \fKone_{N,u} + \fKtwo_{N,u} + \fKthree_{N,u} + \fKfour_{N,u} \, ,
\]
where the first one corresponds to the above mentioned residue:
\begin{align}
\label{eq:fKer1}
\fKone_{N,u}(\lambda,\xi)
&:= - \frac{1}{2\pi\i} \int_{\Gamma_1} \diff z \,
\e^{-(2u + \lambda+\xi) z}
\bigg[\frac{\gamma+z}{\gamma-z}\bigg]^{2N} \, ,
\end{align}
and the other three terms are
\begin{align}
\label{eq:fKer2}
\fKtwo_{N,u}(\lambda,\xi)
&:= \frac{1}{(2\pi\i)^2}
\int_{\Gamma_1} \diff z \,
\int_{\Gamma_2} \diff w \, \e^{-\lambda z - \xi w}
\frac{\e^{-2uz}}{w-z}
\bigg[ \frac{(z+\gamma)(w+\gamma)}{(z-\gamma)(w-\gamma)} \bigg]^N \, , \\
\label{eq:fKer3}
\fKthree_{N,u}(\lambda,\xi)
&:= \fKtwo_{N,u}(\xi,\lambda) \, , \\
\label{eq:fKer4}
\fKfour_{N,u}(\lambda,\xi)
&:= \frac{1}{(2\pi\i)^2}
\int_{\Gamma_1} \diff z \,
\int_{\Gamma_2} \diff w \, \e^{-\lambda z - \xi w}
\frac{\e^{-2u(z+w)}}{z+w}
\bigg[ \frac{(z+\gamma)(w+\gamma)}{(z-\gamma)(w-\gamma)} \bigg]^N \, .
\end{align}

{ \bfseries Step 1: Main contribution in the kernel.} The Airy kernel emerges from a rescaling of $\fKone_{N,u}$ through Proposition~\ref{prop:steepestDescent}, whereas the other terms turn out to be negligible under the same rescaling, as we will see. From now on, fixing $r\in\R$ once for all, we will take $u$ to be $u_N:=2N/\gamma+rN^{1/3}$, as in~\eqref{eq:fAsympt}. Moreover, we denote by $\tilde{\Psi}$ the rescaling of any function $\Psi(\lambda,\xi)$ by the factor $\sqrt[3]{2N}/\gamma$:
\begin{equation}
\label{eq:rescaling}
\tilde{\Psi}(\lambda,\xi)
:= \frac{\sqrt[3]{2N}}{\gamma} \Psi\bigg( \frac{\sqrt[3]{2N}}{\gamma} \lambda, \frac{\sqrt[3]{2N}}{\gamma} \xi \bigg) \, .
\end{equation}
By Proposition~\ref{prop:steepestDescent}, $\fKonetilde_{N,u_N}$ has a non-trivial limit:
\[
\begin{split}
\fKonetilde_{N,u_N}(\lambda,\xi)
= &- \frac{\sqrt[3]{4N}}{\sqrt[3]{2}\gamma}
\frac{1}{2\pi\i} \int_{\Gamma_1} \diff z \,
\exp\bigg\{-z\bigg[\frac{2}{\gamma} 2N +\bigg(\frac{\lambda + \xi}{\sqrt[3]{2}} + 2^{1/3}\gamma r \bigg) \frac{\sqrt[3]{4N}}{\gamma} \bigg]\bigg\}
\bigg[\frac{\gamma+z}{\gamma-z}\bigg]^{2N} \\
\xrightarrow{N\to\infty} 
& \, 2^{-1/3} \Ai\big(2^{-1/3}(\lambda+\xi)+2^{1/3}\gamma r \big) \, .
\end{split}
\]
We thus need to replace our whole kernel with its rescaling by the factor $\sqrt[3]{2N}/\gamma$:
\[
\fKtilde_{N,u_N}
= \fKonetilde_{N,u_N} +\fKtwotilde_{N,u_N} +\fKthreetilde_{N,u_N} + \fKfourtilde_{N,u_N}
\, .
\]
This does not affect formula~\eqref{eq:fFredholm}, as it just amounts to a change of variables in the multiple integrals defining the Fredholm determinant expansion (see~\eqref{eq:fredholmDet}), so that:
\[
\P\big(\fTau_{2N}\leq u_N\big)
= \det(I - \fKtilde_{N,u_N})_{L^2(\R_+)}
\, .
\]

{ \bfseries Step 2: Vanishing terms in the kernel.} We will now show that all the remaining terms $\tilde{K}^{{\rm flat},i}_{N,u_N}(\lambda,\xi)$ for $i=2,3,4$ vanish, starting from $\fKtwotilde_{N,u_N}(\lambda,\xi)$. For this purpose, we specify the contours appropriately. We choose $\Gamma_1$ to be a circle of radius $\rho_1$ around $\gamma$, where $0<\rho_1<\gamma$. Next, we choose $\Gamma_2$ to be a semicircle of radius $\rho_2$ centered at $\delta$, where $0<\delta < \gamma -\rho_1$, composed by concatenating the segment $\delta + \i [-\rho_2,\rho_2]$ and the arc parametrized by $\delta + \rho_2 \e^{\i \theta}$ for $\theta \in [-\pi/2,\pi/2]$. It is clear that both contours lie in the right half-plane and, for $\rho_2$ large enough, $\Gamma_2$ encloses $\Gamma_1$.
Rescaling~\eqref{eq:fKer2}, setting $u:=u_N$, and using the fact that $\lambda,\xi\geq 0$ and $\delta \leq \Re(z),\Re(w) \leq \delta + \rho_2$ for $z\in \Gamma_1$ and $w\in \Gamma_2$, we estimate
\begin{equation}
\begin{split}
\label{eq:fKer2Estimate}
\abs{\fKtwotilde_{N,u_N}(\lambda,\xi)}
&\leq \frac{(2N)^{1/3} \mathcal{L}(\Gamma_1) \mathcal{L}(\Gamma_2)}{\gamma (2\pi)^2 \rm{dist}(\Gamma_1,\Gamma_2)}
\e^{- (\lambda +\xi) \delta \sqrt[3]{2N} / \gamma }
\e^{(m_1 + m_2) N + 2 \abs{r} (\delta + \rho_2) N^{1/3} } \\
&\leq c \, \e^{- (\lambda +\xi) \delta \sqrt[3]{2} / \gamma }
\exp\bigg\{(m_1 + m_2) N + 2 \abs{r} (\delta + \rho_2) N^{1/3} + \frac{1}{3} \log N \bigg\} \, .
\end{split}
\end{equation}
In the first inequality, we have denoted by $\mathcal{L}( \cdot)$ the length of a curve and by $\rm{dist}( \cdot, \cdot)$ the Euclidean distance in $\C$. In the second inequality, $c$ is a constant depending on the parameters $\gamma$, $\delta$, $\rho_1$ and $\rho_2$ only, whereas $m_1$ and $m_2$ are defined by
\begin{align*}
m_1 := \max_{z \in \Gamma_1} \bigg\{ -\frac{4}{\gamma}\Re(z) + 
\log \abs{\frac{z + \gamma}{z - \gamma}} \bigg\}  \, , \qquad\quad
m_2 := \max_{w \in \Gamma_2} \log \abs{\frac{w + \gamma}{w - \gamma}} \, .
\end{align*}
A trivial estimate yields
\[
m_1 \leq -4\bigg(1-\frac{\rho_1}{\gamma} \bigg)
+ \log\bigg(1+ 2 \frac{\gamma}{\rho_1}\bigg) \, .
\]
Now, the function
\[
g(t) := -4(1-t) + \log\Big(1 + \frac{2}{t}\Big)
\]
attains its minimum for $t\in (0,1)$ at $t_0 :=\sqrt{3/2}-1$, with $g(t_0) <0$; hence, choosing $\rho_1 := t_0 \gamma$, we have that $m_1<0$.
On the other hand, we estimate
\[
\begin{split}
m_2
\leq \max_{\Re(w)=\delta} \log \abs{\frac{w + \gamma}{w - \gamma}}
+ \max_{\abs{w-\delta}=\rho_2} \log \abs{\frac{w + \gamma}{w - \gamma}}
% \leq \max_{t\in\R} \bigg\{ \frac{1}{2} \log \frac{(\gamma +\delta)^2 + t^2}{(\gamma-\delta)^2 + t^2} \bigg\} + \max_{\abs{w-\delta}=\rho_2} \log \frac{\abs{w-\delta} + \delta + \gamma}{\abs{w-\delta} - \delta - \gamma}
\leq \log\frac{\gamma+\delta}{\gamma-\delta}
+ \log \frac{\rho_2 +\delta + \gamma}{\rho_2 +\delta - \gamma}  \, .
\end{split}
\]
We can now choose $\delta>0$ small enough and $\rho_2$ big enough such that  $m_2 < -m_1$. It thus follows that, for certain values of $\rho_1$, $\rho_2$ and $\delta$, the quantity after the last inequality in~\eqref{eq:fKer2Estimate} decays exponentially with rate $N$, allowing us to conclude that $\fKtwotilde_{N,u_N}(\lambda,\xi)$ vanishes as $N\to\infty$, for all $\lambda,\xi\in\R_+$. Note that, in~\eqref{eq:fKer2Estimate}, the  exponential containing variables $\lambda$ and $\xi$ does not play any role here, but will provide a useful estimate in the next step.

Because of~\eqref{eq:fKer3}, we have that estimate~\eqref{eq:fKer2Estimate} is exactly valid for $\fKthreetilde_{N,u_N}(\lambda,\xi)$ too, so that this term also vanishes.

Finally, an estimate similar to~\eqref{eq:fKer2Estimate} holds for $\fKfourtilde_{N,u_N}(\lambda,\xi)$: To see this, we make the same contour choice as we made in that case with the aim of showing that $\fKtwotilde_{N,u_N}(\lambda,\xi)$ vanishes. Rescaling~\eqref{eq:fKer4} and setting $u:=u_N$, we then obtain
\begin{equation}
\begin{split}
\label{eq:fKer4Estimate}
\abs{\fKfourtilde_{N,u_N}(\lambda,\xi)}
&\leq \frac{(2N)^{1/3} \mathcal{L}(\Gamma_1) \mathcal{L}(\Gamma_2)}{\gamma (2\pi)^2 \rm{dist}(\Gamma_1,-\Gamma_2)}
\e^{- (\lambda +\xi) \delta \sqrt[3]{2N} / \gamma }
\e^{(m_1 + m_2) N + 4 \abs{r} (\delta + \rho_2) N^{1/3} } \\
&\leq c' \e^{- (\lambda +\xi) \delta \sqrt[3]{2} / \gamma }
\exp\bigg\{(m_1 + m_2) N + 4 \abs{r} (\delta + \rho_2) N^{1/3} + \frac{1}{3} \log N \bigg\} \, ,
\end{split}
\end{equation}
where the constant $c'$ depends on $\gamma$, $\delta$, $\rho_1$ and $\rho_2$ only.
We have already proved that $m_1 + m_2 <0$ for a certain choice of $\rho_1$, $\rho_2$ and $\delta$, hence $\fKfourtilde_{N,u_N}(\lambda,\xi)$ also vanishes.

{ \bfseries Step 3: Convergence of Fredholm determinants.} In the first two steps, we have proven that
\begin{equation}
\label{eq:fKerConvergence}
\lim_{N\to\infty} \fKtilde_{N,u_N}(\lambda,\xi)
= 2^{-1/3} \Ai\big(2^{-1/3}(\lambda+\xi)+ 2^{1/3}\gamma r \big)
\end{equation}
for all $\lambda,\xi\in\R_+$.
We now need to show the convergence of the corresponding Fredholm determinants on $L^2(\R_+)$.
The argument is standard, and based on the series expansion~\eqref{eq:fredholmDet} of the Fredholm determinant.
Notice first that there exist two positive constants $c_1$ and $c_2$ such that
\[
\sup_{N\in\N} \abs{\fKtilde_{N,u_N}(\lambda,\xi)}
\leq c_1 \e^{-c_2 \lambda}
\]
for all $\lambda,\xi\in\R_+$. The exponential bound for $\fKonetilde_{N,u_N}$ comes from Corollary~\ref{coro:steepestDescentEstimate}, whereas the estimates for the remaining terms directly follow from~\eqref{eq:fKer2Estimate} and~\eqref{eq:fKer4Estimate}. Hadamard's bound then implies that
\[
\abs{ \det(\fKtilde_{N,u_N}(\lambda_i, \lambda_j))_{i,j=1}^n }
\leq n^{n/2} \prod_{i=1}^n c_1 \e^{-c_2 \lambda_i } \, .
\]
It follows that
\[
\begin{split}
\abs{ \det(I - \fKtilde_{N,u_N})_{L^2(\R_+)} }
&\leq 1 + \sum_{n=1}^{\infty} \frac{1}{n!} \int_0^{\infty} \dots \int_0^{\infty} \abs{ \det(K_N(\lambda_i, \lambda_j))_{i,j=1}^n } \diff \lambda_1 \cdots \lambda_n \\
& \leq 1 + \sum_{n=1}^{\infty} \frac{n^{n/2}}{n!} \bigg( \int_{0}^{\infty} c_1 \e^{-c_2 \lambda} \diff \lambda \bigg)^n
< \infty \, .
\end{split}
\]
These inequalities, apart from providing a further proof that the Fredholm determinants of our kernels are well-defined, allow us to conclude, by dominated convergence, that limit~\eqref{eq:fKerConvergence} still holds when passing to the corresponding Fredholm determinants on $L^2(\R_+)$. 
After a rescaling of the limiting kernel by a factor $2^{-2/3}$, one can see that the operator on $L^2(\R_+)$ defined through the kernel $(\lambda,\xi) \mapsto 2^{-1/3} \Ai\big(2^{-1/3}(\lambda+\xi)+s\big)$ has the same Fredholm determinant as the operator $\mathcal{K}_1$ on $L^2([s,\infty))$ defining the Tracy-Widom GOE distribution $F_1(s)$ as in~\eqref{eq:GOE}. This concludes the proof.
\end{proof}

\section{Point-to-half-line last passage percolation and ${\rm Airy}_{2\to 1}$}
\label{sec:hLPP}

In this section, we analyze the point-to-half-line dLPP, using again the general results of sections~\ref{sec:fredholm} and~\ref{sec:steepestDescent}. 
The starting point is the next theorem, which is a result of~\eqref{eq:contS}, \eqref{eq:contSp} and~\eqref{eq:hLPPint}, combined with the Cauchy-Binet identity.
\begin{theorem}[\cite{BZ19a}]
\label{thm:hLPPdet}
The distribution function of the point-to-half-line directed last passage percolation $\hTau_{2N}$ with exponential waiting times as in~\eqref{eq:expDistribution} is a ratio of $N\times N$ determinants:
\begin{align}
\label{eq:hLPPdet}
\P\big(\hTau_{2N}\leq u\big) &= 
\frac{\det\big(\hH_u(\alpha_i,\beta_j)\big)_{1\leq i,j\leq N}}{\det(C(\alpha_i,\beta_j))_{1\leq i,j\leq N}} \, ,
\end{align}
for $u>0$, where $C(z,w) := (z+w)^{-1}$ and
\begin{align*}
\hH_u(z,w) &:= \e^{-u(z + w)} \int_0^u (\e^{z x}-\e^{-z x})
\e^{w x} \diff x \, .
\end{align*}
\end{theorem}
 We now write the distribution function as a Fredholm determinant.
\begin{theorem}
\label{thm:hFredholm}
The distribution of $\hTau_{2N}$ with exponential waiting times as in~\eqref{eq:expDistribution} is given by
\begin{equation}
\label{eq:hFredholm}
\P\big(\hTau_{2N}\leq u\big) = \det(I - \hK_{N,u})_{L^2(\R_+)} \, ,
\end{equation}
where $\hK_{N,u} \colon L^2(\R_+)\to L^2(\R_+)$ is the operator defined through the kernel
\begin{equation}
\label{eq:hKernel}
\hK_{N,u}(\lambda,\xi)
= \frac{1}{(2\pi\i)^2}
\int_{\Gamma_1} \diff z
\int_{\Gamma_2} \diff w \,
\e^{-\lambda z - \xi w} \,
\hHbar_u(z,w)
\prod_{m=1}^N \frac{(z+\beta_m)(w+\alpha_m)}{(z-\alpha_m)(w-\beta_m)} \, .
\end{equation}
Here, $\Gamma_1,\Gamma_2 \subset \C_+$ are any positively oriented simple closed contours such that $\Gamma_1$ encloses $\alpha_1,\dots,\alpha_N$, $\Gamma_2$ encloses $\beta_1,\dots,\beta_N$ as well as the whole $\Gamma_1$, and
\begin{equation}
\label{eq:hHbar}
\hHbar_u(z,w)
= \frac{\e^{-u(z+w)}}{z+w}
+ \frac{\e^{-u(z+w)}}{z-w}
+ \frac{\e^{-2uz}}{w-z} \, .
\end{equation}
\end{theorem}

\begin{proof}
The claim follows from formula~\eqref{eq:hLPPdet} and Theorem~\ref{thm:detToFredDet}. Function $\hHbar_u$ in~\eqref{eq:hKernel} is defined through the relation $\hH_u = C - \hHbar_u$ (using the notation of Theorem~\ref{thm:hLPPdet} for $\hH_u$),  i.e.
\[
\hHbar_u(z,w)
= \frac{1}{z+w} - \e^{-u(z + w)} \int_0^u (\e^{z x}-\e^{-z x})
\e^{w x} \diff x \, .
\]
If we assume that $\Gamma_2$ encloses $\Gamma_1$ (so that $z\neq w$ for all $z,w$), integrating the above expression yields~\eqref{eq:hHbar}.
\end{proof}

\begin{proof}[{\bf Proof of Theorem~\ref{thm:hAsympt}}]
In order to perform the asymptotics of formula~\eqref{eq:hFredholm} in the i.i.d.\ case, we set $\alpha_m=\beta_m=\gamma$ for all $m$ in parametrization~\eqref{eq:expDistribution}. Our kernel~\eqref{eq:hKernel} thus becomes
\[
\begin{split}
\hK_{N,u}
= \hKone_{N,u} + \hKtwo_{N,u} + \hKthree_{N,u} \, ,
\end{split}
\]
where
\begin{align*}
\hKone_{N,u}(\lambda,\xi)
&= \frac{1}{(2\pi\i)^2}
\int_{\Gamma_1} \diff z
\int_{\Gamma_2} \diff w \,
\e^{-\lambda z - \xi w}
\frac{\e^{-u(z+w)}}{z+w}
\bigg[ \frac{(z+\gamma)(w+\gamma)}{(z-\gamma)(w-\gamma)} \bigg]^N
\, , \\
\hKtwo_{N,u}(\lambda,\xi)
&= \frac{1}{(2\pi\i)^2}
\int_{\Gamma_1} \diff z
\int_{\Gamma_2} \diff w \,
\e^{-\lambda z - \xi w}
\frac{\e^{-u(z+w)}}{z-w}
\bigg[ \frac{(z+\gamma)(w+\gamma)}{(z-\gamma)(w-\gamma)} \bigg]^N \, , \\
\hKthree_{N,u}(\lambda,\xi)
&= \frac{1}{(2\pi\i)^2}
\int_{\Gamma_1} \diff z
\int_{\Gamma_2} \diff w \,
\e^{-\lambda z - \xi w}
\frac{\e^{-2uz}}{w-z}
\bigg[ \frac{(z+\gamma)(w+\gamma)}{(z-\gamma)(w-\gamma)} \bigg]^N \, .
\end{align*}

For the steepest descent analysis of the first two terms, we are going to adapt the proof of Proposition~\ref{prop:steepestDescent}, taking into account that we now have double contour integrals instead of single ones.
Noticing that $\hKone_{N,u}$ and $\hKtwo_{N,u}$ only differ for the sign in $(z\pm w)^{-1}$, we study both at the same time, denoting by $K^{\pm}_{N,u}$ either of them:
\[
K^{\pm}_{N,u}(\lambda,\xi)
:= \frac{1}{(2\pi\i)^2}
\int_{\Gamma_1} \diff z
\int_{\Gamma_2} \diff w \,
\e^{-\lambda z - \xi w}
\frac{\e^{-u(z+w)}}{z \pm w}
\bigg[ \frac{(z+\gamma)(w+\gamma)}{(z-\gamma)(w-\gamma)} \bigg]^N
\, .
\]
We replace the contour $\Gamma_1$ with $T_{R_1} + 2\epsilon\gamma/\sqrt[3]{2N}$ and the contour $\Gamma_2$ with $T_{R_2} + \epsilon\gamma/\sqrt[3]{2N}$, for some $\gamma < R_1 < R_2$ and $\epsilon>0$; Here, as in the proof of Proposition~\ref{prop:steepestDescent}, $T_R$ is the negatively oriented triangular path with vertices $0$, $2R\e^{\i\pi/3}$ and $2R\e^{-\i\pi/3}$.
Notice that changing the orientation of both paths does not yield any change of sign in the double contour integral; moreover, the first contour is still enclosed by the second one, and the singularities at $(z\pm w)^{-1}$ are not crossed by the deformed contours (the infinitesimal shifts of $T_{R_1}$ and $T_{R_2}$ are also done for this sake). 
Set now $u=u_N:= 2N/\gamma + r N^{1/3}$ and denote by $\tilde{\Psi}$ the rescaling of any function $\Psi(\lambda,\xi)$ by the factor $\sqrt[3]{2N}/\gamma$, as in~\eqref{eq:rescaling}.
We can thus write
\[
\tilde{K}^{\pm}_{N,u_N}(\lambda, \xi)
= \frac{\sqrt[3]{2N}}{\gamma (2\pi\i)^2}
\int_{T_{R_1} + \frac{2\epsilon\gamma}{\sqrt[3]{2N}}} \diff z
\int_{T_{R_2} + \frac{\epsilon\gamma}{\sqrt[3]{2N}}} \diff w
\frac{ \e^{N F(z) - \lambda_r z \sqrt[3]{2N}/\gamma }
 \e^{ N F(w) - \xi_r w \sqrt[3]{2N}/\gamma} }{z \pm w} \, ,
\]
where $\lambda_r := \lambda + 2^{-1/3}\gamma r$, $\xi_r := \xi + 2^{-1/3}\gamma r$, and $F(z):= \log(\gamma+z) - \log(\gamma - z) - 2z/\gamma$. Since the main contribution in the integral is expected to come from $z=w=0$, which is the critical point of $F$, we split the above integral into the following sum:
\[
\tilde{K}^{\pm}_{N,u_N}
= \tilde{K}^{\pm, \rm{in, in}}_{N,u_N}
+ \tilde{K}^{\pm, \rm{in, ex}}_{N,u_N}
+ \tilde{K}^{\pm, \rm{ex, in}}_{N,u_N}
+ \tilde{K}^{\pm, \rm{ex, ex}}_{N,u_N} \, .
\]
Here, the first superscript ``$\rm{in}$'' (``$\rm{ex}$'') indicates that the integration w.r.t.\ $z$ is performed only in the interior (exterior, respectively) of the ball $\{ |z - 2\epsilon\gamma/\sqrt[3]{2N}| \leq \gamma N^{-\alpha}\}$ for some exponent $\alpha>0$ to be specified later on, while the second superscript ``$\rm{in}$'' (``$\rm{ex}$'') indicates that the integration w.r.t.\ $w$ is performed only in the interior (exterior, respectively) of the ball $\{ |w - \epsilon\gamma/\sqrt[3]{2N}| \leq \gamma N^{-\alpha}\}$. In~$\tilde{K}^{\pm, \rm{in, in}}_{N,u_N}$, after the changes of variables $\tilde{z} := z \sqrt[3]{2N}/\gamma$ and $\tilde{w} := w \sqrt[3]{2N}/\gamma$, we obtain
\[
\begin{split}
\tilde{K}^{\pm, \rm{in, in}}_{N,u_N}(\lambda,\xi)
= &\frac{1}{(2\pi\i)^2}
\int_{C + 2\epsilon} \diff \tilde{z}
\exp\bigg\{ \frac{\tilde{z}^3}{3} - \lambda_r \tilde{z} +  R\bigg( \frac{\gamma}{\sqrt[3]{2N}} \tilde{z}\bigg) N  \bigg\}
\1_{\big\{ \abs{ \tilde{z} - 2\epsilon} \leq \sqrt[3]{2} N^{1/3-\alpha} \big\}} \\
&\times \int_{C + \epsilon} \diff \tilde{w}
\exp\bigg\{ \frac{\tilde{w}^3}{3} - \xi_r \tilde{w} +  R\bigg( \frac{\gamma}{\sqrt[3]{2N}} \tilde{w}\bigg) N \bigg\}
\1_{\{\abs{\tilde{w} - \epsilon} \leq \sqrt[3]{2} N^{1/3-\alpha} \}}
\frac{1}{\tilde{z} \pm \tilde{w}} \, ,
\end{split}
\]
where $C$ is the piecewise linear path going from $\e^{-\i\pi/3}\infty$ to the origin to $\e^{\i\pi/3}\infty$, and $R$ is the remainder defined by~\eqref{eq:expansionCriticalPoint}. The indicator functions clearly converge to $1$ if $\alpha < 1/3$. As in the proof of Proposition~\ref{prop:steepestDescent}, one can also show that the remainders, even when multiplied by $N$, vanish uniformly for $\tilde{z},\tilde{w}$ in the support of the integrand, if we choose $1/4< \alpha < 1/3$. Applying dominated convergence, one can see that
\[
\lim_{N\to\infty} \tilde{K}^{\pm, \rm{in, in}}_{N,u_N}(\lambda,\xi)
= \frac{1}{(2\pi\i)^2}
\int_{C + 2\epsilon} \diff \tilde{z}
\int_{C + \epsilon} \diff \tilde{w}
\frac{1}{\tilde{z} \pm \tilde{w}}
\exp\bigg\{ \frac{\tilde{z}^3}{3} - \lambda_r \tilde{z}
+ \frac{\tilde{w}^3}{3} - \xi_r \tilde{w} \bigg\} \, .
\]
Using similar arguments as in the proof of Proposition~\ref{prop:steepestDescent}, together with the bound $\abs{z \pm w}^{-1} \leq 2\sqrt[3]{2N}/(\sqrt{3}\epsilon\gamma)$, one can see that the other terms $\tilde{K}^{\pm, \rm{in, ex}}_{N,u_N}$, $\tilde{K}^{\pm, \rm{ex, in}}_{N,u_N}$, and $\tilde{K}^{\pm, \rm{ex, ex}}_{N,u_N}$ vanish exponentially fast in the limit $N\to\infty$. We thus have:
\begin{align*}
\lim_{N\to\infty} \hKonetilde_{N,u_N}(\lambda,\xi)
&= \frac{1}{(2\pi\i)^2}
\int_{C + 2\epsilon} \diff \tilde{z}
\int_{C + \epsilon} \diff \tilde{w}
\frac{1}{\tilde{z} + \tilde{w}}
\e^{ \tilde{z}^3/3 - \lambda_r \tilde{z}}
\e^{\tilde{w}^3/3 - \xi_r \tilde{w}} \, , \\
\lim_{N\to\infty} \hKtwotilde_{N,u_N}(\lambda,\xi)
&= \frac{1}{(2\pi\i)^2}
\int_{C + 2\epsilon} \diff \tilde{z}
\int_{C + \epsilon} \diff \tilde{w}
\frac{1}{\tilde{z} - \tilde{w}}
\e^{ \tilde{z}^3/3 - \lambda_r \tilde{z}}
\e^{\tilde{w}^3/3 - \xi_r \tilde{w}} \, .
\end{align*}
We will now rewrite these expressions as integrals of Airy functions. In the first one, since $\Re(\tilde{z}+\tilde{w}) >0$ for all $\tilde{z}$ and $\tilde{w}$, we can make the substitution $(\tilde{z}+\tilde{w})^{-1} = \int_0^{\infty} \e^{-(\tilde{z}+\tilde{w})x} \diff x$. The resulting $\diff \tilde{z} \diff \tilde{w} \diff x$  integral is absolutely convergent, hence Fubini's Theorem can be applied to obtain:
\[
\begin{split}
\lim_{N\to\infty} \hKonetilde_{N,u_N}(\lambda,\xi)
&= \int_0^{\infty}
\bigg[ \frac{1}{2\pi\i}
\int_{C + 2\epsilon}
\e^{ \tilde{z}^3/3 - (\lambda_r +x) \tilde{z}} \diff \tilde{z} \bigg]
\bigg[ \frac{1}{2\pi\i}
\int_{C + \epsilon}
\e^{\tilde{w}^3/3 - (\xi_r +x) \tilde{w}} \diff \tilde{w} \bigg] \diff x \\
&= \int_0^{\infty} \Ai(\lambda_r + x) \Ai(\xi_r + x) \diff x \, ,
\end{split}
\]
according to definition~\ref{eq:Airy} of Airy function. In $\hKtwotilde_{N,u_N}$, we deform the contour $C + \epsilon$ into the straight line $l_{\epsilon}$ going from $\epsilon-\i\infty$ to $\epsilon+\i\infty$; since now $\Re(\tilde{z}-\tilde{w})\geq \epsilon >0$ for all $\tilde{z}$ and $\tilde{w}$, we can make the substitution $(\tilde{z}-\tilde{w})^{-1} = \int_0^{\infty} \e^{-(\tilde{z}-\tilde{w})x} \diff x$. The resulting $\diff \tilde{z} \diff \tilde{w} \diff x$  integral is absolutely convergent, hence Fubini's Theorem can be applied to obtain:
\[
\begin{split}
\lim_{N\to\infty} \hKtwotilde_{N,u_N}(\lambda,\xi)
&= \int_0^{\infty}
\bigg[ \frac{1}{2\pi\i}
\int_{C + 2\epsilon}
\e^{ \tilde{z}^3/3 - (\lambda_r +x) \tilde{z}} \diff \tilde{z} \bigg]
\bigg[ \frac{1}{2\pi\i}
\int_{l_{\epsilon}}
\e^{\tilde{w}^3/3 - (\xi_r - x) \tilde{w}} \diff \tilde{w} \bigg] \diff x \\
&= \int_0^{\infty} \Ai(\lambda_r + x) \Ai(\xi_r - x) \diff x \, .
\end{split}
\]
We remark that the second square bracket above is an Airy function as well, since the path $l_{\epsilon}$ can be deformed back to a contour, like $C+\epsilon$, whose arguments at $\infty$ are $\pm\pi/3$.

We finally notice that $\hKthree_{N,u_N}(\lambda,\xi)$ equals exactly the term $\fKtwo_{N,u_N}(\lambda,\xi)$ defined in the proof of Theorem~\ref{thm:fAsympt}. Therefore, as we have already proved there, the rescaled version $\hKthreetilde_{N,u_N}(\lambda,\xi)$ vanishes as $N\to\infty$.

We conclude that, as a whole, our rescaled kernel has the following limit:
\[
\lim_{N\to\infty} \hKtilde_{N,u_N}(\lambda,\xi)
= \mathcal{K}_{2\to 1}(\lambda_r, \xi_r )
= \mathcal{K}_{2\to 1}(\lambda + 2^{-1/3}\gamma r, \xi + 2^{-1/3}\gamma r ) \, ,
\]
where $\mathcal{K}_{2\to 1}$ is defined in~\eqref{eq:Airy21kernel}.  Using the key fact that all contours are chosen to have positive distance from the imaginary axis (as in the analogous estimates obtained in Corollary~\ref{coro:steepestDescentEstimate} and in the proof of Theorem~\ref{thm:fAsympt}), one can show that there exist two positive constants $c_1$ and $c_2$ such that, for all $\lambda,\xi\in\R_+$,
\[
\sup_{N\in\N} \big| \hKtilde_{N,u_N}(\lambda,\xi) \big|
\leq c_1 \e^{-c_2 \lambda} \, .
\]
The latter bound provides, as in the third step of the proof of Theorem~\ref{thm:fAsympt}, the right estimates for the series expansion of $\det(I - \hKtilde_{N,u_N})_{L^2(\R_+)}$ in order to justify its convergence. 
The claim thus follows from the Fredholm determinant representation~\eqref{eq:Airy21} of $F_{2\to 1}$.
\end{proof}

{\bf Acknowledgements}: 
This article is dedicated to Raghu Varadhan on the occasion of his 75th birthday. Exact solvability is arguably not
within Raghu's signature style. However, the second author learned about random polymer models and the $t^{1/3}$ law from Raghu, as 
being a challenging problem, while he was working under his direction on a somewhat different PhD topic. Since that time 
he has always had the desire to make some contributions in this direction and would like to thank him for, among other things,
having provided this stimulus.

The work of EB was supported by EPSRC via grant EP/M506679/1.
The work of NZ was supported by EPSRC via grant EP/L012154/1.


\vskip 4mm
\begin{thebibliography}{COSZ14}

\bibitem[Bis18]{Bis18}
E. Bisi.
\newblock {\em Random polymers via orthogonal Whittaker and symplectic Schur functions}.
\newblock PhD thesis, University of Warwick, 2018. arXiv:1810.03734.

\bibitem[BR01]{BR01}
J.~Baik and E.M.~Rains. 
\newblock The asymptotics of monotone subsequences of involutions.
\newblock {\em Duke Math.\ Journal}, 109(2):205--281, 2001.

\bibitem[BFS08]{BFS08}
A.~Borodin, P.L.~Ferrari, and T.~Sasamoto.
\newblock Transition between Airy1 and Airy2 processes and TASEP fluctuations.
\newblock {\em Comm.\ Pure Appl.\ Math.}, 61:1603--1629, 2008.

\bibitem[BG12]{BG12}
A.~Borodin and V.~Gorin.
\newblock Lectures on integrable probability.
\newblock {\em Probability and Statistical Physics in St. Petersburg, Proceedings of Symposia in Pure Mathematics}, Vol.~91, 2012.

\bibitem[BZ19a]{BZ19a}
E.~Bisi and N.~Zygouras.
\newblock Point-to-line polymers and orthogonal Whittaker functions.
\newblock {\em Trans.\ Am.\ Math.\ Soc.}, 371(12):8339-8379, 2019.

\bibitem[BZ19b]{BZ19b}
E.~Bisi and N.~Zygouras.
\newblock Transition between characters of classical groups, decomposition of Gelfand-Tsetlin patterns and last passage percolation.
\newblock arXiv:1905.09756.

\bibitem[COSZ14]{COSZ14}
I.~Corwin, N.~{O'Connell}, T.~{Sepp\"al\"ainen}, and N.~Zygouras.
\newblock Tropical combinatorics and {Whittaker} functions.
\newblock {\em Duke Math.\ J.}, 163(3):513--563, 2014.

\bibitem[F04]{F04}
P.L.~Ferrari.
\newblock Polynuclear growth on a flat substrate and edge scaling of GOE eigenvalues. 
\newblock {\em Comm.\ Math.\ Phys.}, 252:77--109, 2004.

\bibitem[FS05]{FS05}
P.L.~Ferrari and H.~Spohn.
\newblock A determinantal formula for the GOE Tracy-Widom distribution.
\newblock {\em J.\ Phys.\ A: Math.\ Gen.}, 38:L557--L561, 2005.

\bibitem[FH91]{FH91}
W.~Fulton and J.~Harris.
\newblock Representation Theory - A First Course.
\newblock Vol.~129 of {\em Graduate Texts in Mathematics}, Springer-Verlag, 1991.

\bibitem[F10]{F10}
P.J.~Forrester.
\newblock Log-Gases and Random Matrices (LMS-34).
\newblock{\em Princeton University Press}, 2010.

\bibitem[J00]{J00}
K.~Johansson.
\newblock Shape Fluctuations and Random Matrices.
\newblock{\em Commun.\ Math.\ Phys.}, 209(2):437--476, 2000.

\bibitem[J01]{J01}
K.~Johansson.
\newblock Random growth and random matrices.
\newblock{\em European Congress of Mathematics}, 445--456, 2001.

\bibitem[K01]{K01}
A.N.~Kirillov.
\newblock Introduction to tropical combinatorics. 
\newblock {\em Physics and Combinatorics. Proc.\
Nagoya 2000 2nd Internat.\ Workshop} (A.N.~Kirillov and N.~Liskova, eds.), World Scientific,
Singapore, 82--150, 2001.

\bibitem[K05]{K05}
W.~K\"onig.
\newblock Orthogonal polynomial ensembles in probability theory.
\newblock{\em Probab. Surveys}, (2):385--447, 2005.

\bibitem[LeDC12]{LeDC12}
P.~Le Doussal and P.~Calabrese. 
\newblock The KPZ equation with flat initial condition and the directed polymer with one free end.
\newblock{\em J. Stat. Mech.: Theory and Experiment}, 06:P06001, 2012.

\bibitem[NZ17]{NZ17}
V.-L.~Nguyen and N.~Zygouras.
\newblock Variants of geometric {RSK}, geometric {PNG} and the multipoint distribution of the log-{Gamma} polymer.
\newblock {\em  Int.\ Math.\ Res.\ Not.}, (15):4732--4795, 2017.

\bibitem[OSZ14]{OSZ14}
N.~{O'Connell}, T.~{Sepp\"al\"ainen}, and N.~Zygouras.
\newblock Geometric {RSK} correspondence, {Whittaker} functions and symmetrized random polymers.
\newblock {\em Inventiones Math.}, 197(2):361--416, 2014.

\bibitem[O01]{O01}
A.~Okounkov. 
\newblock Infinite wedge and random partitions.
\newblock {\em Selecta Mathematica}, New Series 7.1, 57--81, 2001.

\bibitem[Sa05]{Sa05}
T.~Sasamoto.
\newblock Spatial correlations of the 1D KPZ surface on a flat substrate.
\newblock {\em J.\ Phys.\ A: Math.\ Gen.}, 38:L549, 2005.

\bibitem[{Sep}12]{Sep12}
T.~{Sepp\"al\"ainen}.
\newblock Scaling for a one-dimensional directed polymer with boundary
  conditions.
\newblock {\em Annals of Probability}, 40(1):19--73, 2012.

\bibitem[S97]{S97}
G.M.~Sch\"utz.
\newblock Exact solution of the master equation for the asymmetric exclusion process.
 \newblock {\em J.\ Stat.\ Phys.}, 88(1):427--445, 1997.

\bibitem[TW96]{TW96}
C.A.~Tracy and H.~Widom.  
\newblock On orthogonal and symplectic matrix ensembles. 
\newblock {\em Comm.\ Math.\ Phys.}, 177(3):727--754, 1996.

\bibitem[TW98]{TW98}
C.A.~Tracy and H.~Widom.
\newblock Correlation functions, cluster functions, and spacing distributions for random matrices.
\newblock {\em J.\ Stat.\ Phys.}, 92(5):809--835, 1998.

\bibitem[TW05]{TW05}
C.A.~Tracy and H.~Widom. 
\newblock Matrix kernels for the Gaussian orthogonal and symplectic ensembles.
\newblock {\em Annales de l'Institut Fourier}, 55(6):2197--2207, 2005.


\end{thebibliography}
\end{document}